\documentclass[11pt]{article}

\usepackage[a4paper,top=2.35cm,bottom=2.5cm,left=2.45cm,right=2.45cm,
            headheight=22pt,headsep=0.55cm]{geometry}
\usepackage[T1]{fontenc}
\usepackage[utf8]{inputenc}
\usepackage{amsmath,amsthm,mathtools}
\usepackage{newtxtext}
\usepackage{newtxmath}
\usepackage[scaled=0.94]{sourcesanspro}
\usepackage{microtype}

\usepackage{bm}
\usepackage{mathrsfs}

\usepackage{graphicx}
\usepackage{subcaption}
\usepackage{booktabs}
\usepackage{array}
\usepackage{tabularx}
\usepackage{enumitem}
\usepackage{xcolor}
\usepackage{titlesec}
\usepackage{fancyhdr}
\usepackage[most]{tcolorbox}
\usepackage{xurl}
\usepackage{hyperref}
\usepackage[nameinlink,noabbrev]{cleveref}
\usepackage[numbers,sort&compress]{natbib}

\definecolor{PaperNavy}{HTML}{17324D}
\definecolor{PaperTeal}{HTML}{1E6F72}
\definecolor{PaperInk}{HTML}{22272B}
\definecolor{PaperMuted}{HTML}{66727C}
\definecolor{PaperLine}{HTML}{D9E0E4}
\definecolor{PaperWash}{HTML}{F4F7F8}

\color{PaperInk}
\setlist{nosep,leftmargin=2.0em,topsep=0.45em}
\setlist[enumerate,1]{label=\textcolor{PaperTeal}{\arabic*.}}
\setlist[itemize,1]{label=\textcolor{PaperTeal}{\raisebox{0.2ex}{\small$\bullet$}}}

\titleformat{\section}
  {\Large\sffamily\bfseries\color{PaperNavy}}
  {\textcolor{PaperTeal}{\thesection}}{0.72em}{}
  [\vspace{0.25em}\color{PaperLine}\titlerule]
\titleformat{\subsection}
  {\large\sffamily\bfseries\color{PaperNavy}}
  {\textcolor{PaperTeal}{\thesubsection}}{0.68em}{}
\titleformat{\subsubsection}
  {\normalsize\sffamily\bfseries\color{PaperNavy}}
  {\textcolor{PaperTeal}{\thesubsubsection}}{0.65em}{}
\titlespacing*{\section}{0pt}{2.0em}{0.8em}
\titlespacing*{\subsection}{0pt}{1.5em}{0.45em}
\titlespacing*{\subsubsection}{0pt}{1.1em}{0.35em}

\hypersetup{
  colorlinks=true,
  linkcolor=PaperTeal,
  citecolor=PaperTeal,
  urlcolor=PaperTeal,
  pdfborder={0 0 0},
  pdftitle={persistence pairing Graphs: Tracking Basis Transitions in Persistent Homology},
  pdfauthor={John Rick Manzanares}
}
\renewcommand{\headrulewidth}{0.35pt}
\renewcommand{\headrule}{\hbox to\headwidth{\color{PaperLine}\leaders\hrule height \headrulewidth\hfill}}

\newtheoremstyle{paperdefinition}
  {0.75em}{0.75em}{\normalfont}{}%
  {\sffamily\bfseries\color{PaperNavy}}{.}{0.55em}{}
\newtheoremstyle{paperplain}
  {0.75em}{0.75em}{\itshape}{}%
  {\sffamily\bfseries\color{PaperNavy}}{.}{0.55em}{}
\newtheoremstyle{paperremark}
  {0.65em}{0.65em}{\normalfont}{}%
  {\sffamily\bfseries\color{PaperNavy}}{.}{0.55em}{}

\theoremstyle{paperdefinition}
\newtheorem{definition}{Definition}[section]
\theoremstyle{paperplain}
\newtheorem{theorem}[definition]{Theorem}
\newtheorem{proposition}[definition]{Proposition}
\newtheorem{lemma}[definition]{Lemma}
\newtheorem{corollary}[definition]{Corollary}
\theoremstyle{paperremark}

\tcolorboxenvironment{definition}{
  enhanced,breakable,colback=PaperWash,colframe=PaperLine,
  boxrule=0.45pt,arc=1.2pt,left=9pt,right=9pt,top=6pt,bottom=6pt,
  borderline west={2pt}{0pt}{PaperTeal}
}
\tcolorboxenvironment{theorem}{
  enhanced,breakable,colback=white,colframe=PaperLine,
  boxrule=0.45pt,arc=1.2pt,left=9pt,right=9pt,top=6pt,bottom=6pt,
  borderline west={2pt}{0pt}{PaperNavy}
}
\tcolorboxenvironment{proposition}{
  enhanced,breakable,colback=white,colframe=PaperLine,
  boxrule=0.45pt,arc=1.2pt,left=9pt,right=9pt,top=6pt,bottom=6pt,
  borderline west={2pt}{0pt}{PaperNavy}
}
\tcolorboxenvironment{lemma}{
  enhanced,breakable,colback=white,colframe=PaperLine,
  boxrule=0.45pt,arc=1.2pt,left=9pt,right=9pt,top=6pt,bottom=6pt,
  borderline west={2pt}{0pt}{PaperNavy}
}
\tcolorboxenvironment{corollary}{
  enhanced,breakable,colback=white,colframe=PaperLine,
  boxrule=0.45pt,arc=1.2pt,left=9pt,right=9pt,top=6pt,bottom=6pt,
  borderline west={2pt}{0pt}{PaperNavy}
}

\newtcolorbox{intuition}[1][]{
  enhanced,breakable,colback=PaperWash,colframe=PaperLine,
  boxrule=0.4pt,arc=1.2pt,left=10pt,right=10pt,top=7pt,bottom=7pt,
  borderline west={2pt}{0pt}{PaperTeal},
  fonttitle=\sffamily\bfseries\color{PaperNavy},
  title={Intuition\if\relax\detokenize{#1}\relax\else: #1\fi}
}
\newtcolorbox{takeaway}[1][]{
  enhanced,breakable,colback=white,colframe=PaperLine,
  boxrule=0.4pt,arc=1.2pt,left=10pt,right=10pt,top=7pt,bottom=7pt,
  borderline west={2pt}{0pt}{PaperNavy},
  fonttitle=\sffamily\bfseries\color{PaperNavy},
  title={Key point\if\relax\detokenize{#1}\relax\else: #1\fi}
}

\newcommand{\R}{\mathbb{R}}

\newcommand{\safeincludegraphics}[2][]{%
  \IfFileExists{#2}{\includegraphics[#1]{#2}}{%
    \begingroup
    \setlength{\fboxsep}{0pt}%
    \color{PaperLine}%
    \fbox{\parbox[c][4.2cm][c]{0.94\linewidth}{%
      \centering\sffamily\small\color{PaperMuted}
      Figure placeholder\\[0.35em]
      \texttt{\detokenize{#2}}}}
    \endgroup}}

\makeatletter
\title{Persistence Pairing Graphs: Tracking Changes of Selected Homology Bases}

\author{John Rick Manzanares\textsuperscript{1,2,*}}
\date{\today}

\newcommand{\papertype}{Preprint}
\newcommand{\affiliations}{%
\begin{tabularx}{\linewidth}{@{}>{\raggedleft\arraybackslash\sffamily\bfseries\color{PaperTeal}}p{1.5em} X@{}}
1 & Dioscuri Centre in Topological Data Analysis, Institute of Mathematics of the Polish Academy of Sciences, 00-656 Warsaw, Poland\\[0.2em]
2 & International Environmental Doctoral School, University of Silesia in Katowice, 41-200 Sosnowiec, Poland
\end{tabularx}}
\newcommand{\correspondence}{%
John Rick Manzanares, \href{mailto:jdolormanzanares@impan.pl}{jdolormanzanares@impan.pl}}
\newcommand{\funding}{%
This project has received funding from the European Union's Horizon Europe research and innovation programme under the Marie Skłodowska-Curie Actions (MSCA), grant agreement No.~101120290 (GAP).}
\newcommand{\competinginterests}{The authors declare no competing interests.}
\newcommand{\keywords}{cellular homology; persistent cohomology; filtered chain complexes; homological algebra; Gromov--Wasserstein discrepancy; optimal transport}

\renewcommand{\maketitle}{%
  \thispagestyle{empty}
  \noindent
  \begin{minipage}[t]{0.69\textwidth}
    {\sffamily\footnotesize\bfseries\color{PaperTeal}\MakeUppercase{\papertype}}\\[0.75em]
    {\fontsize{23}{27}\selectfont\sffamily\bfseries\color{PaperNavy}\@title\par}
    \vspace{0.85em}
    {\large\sffamily\color{PaperInk}\@author\par}
  \end{minipage}%
  \hfill
  \begin{minipage}[t]{0.24\textwidth}
    \raggedleft
    {\sffamily\footnotesize\color{PaperMuted}Version date\\[0.2em]
    \color{PaperInk}\@date}
  \end{minipage}

  \vspace{1.0em}
  \noindent{\color{PaperTeal}\rule{\textwidth}{1.15pt}}
  \vspace{0.95em}

  {\small\noindent\affiliations\par}
  \vspace{0.75em}

  \noindent\begin{tcolorbox}[
    enhanced,colback=PaperWash,colframe=PaperLine,boxrule=0.4pt,
    arc=1.2pt,left=10pt,right=10pt,top=8pt,bottom=8pt]
  \noindent\begin{tabularx}{\linewidth}{@{}>{\sffamily\bfseries\color{PaperNavy}}p{3.3cm} X@{}}
  Correspondence & \correspondence\\[0.32em]
  Funding & \funding\\[0.32em]
  Competing interests & \competinginterests\\[0.32em]
  Keywords & \keywords
  \end{tabularx}
  \end{tcolorbox}
  \vspace{0.35em}
}
\makeatother

\begin{document}

\maketitle

\begin{abstract}
\noindent
Persistent homology summarizes a filtration by recording when homological features appear and disappear, but the resulting barcode does not determine a preferred homology basis or preferred cycle representatives. We study the additional information carried by a basis selected through matrix reduction and introduce the \emph{persistence pairing graph}, a reduction-dependent directed graph that records how selected homology classes are re-expressed when their associated persistence intervals end. After fixing a filtered finite CW complex, a coefficient field, a filtration-compatible cell order, and a reduced factorization of the boundary matrix, the reduction determines selected birth cycles. When an associated persistence interval ends, the image of its selected birth class has unique coordinates in the basis formed by selected classes whose intervals survive beyond the same filtration value, and the nonzero coordinates define the graph edges. These coordinates can also be recovered through the Kronecker pairing with a dual cohomology basis. We prove that the active selected birth cycles form a homology basis at every cellwise filtration stage, characterize the kernel and image of each homological transition, and show that the graph is directed and acyclic, with every edge pointing toward an interval containing the source interval. In dimension zero, under the standard reduction convention, the graph agrees with the elder rule merge tree. We also prove that its edges can change while the persistence pairs and barcode remain fixed. Finally, we define a bounded fused graph discrepancy and separate the contribution controlled by ordinary persistence stability from the reduction-dependent graph structure. The graph is therefore a descriptor relative to a fixed reduction convention, rather than an invariant of the persistence module alone.
\end{abstract}

\section{Introduction}

Persistent homology studies how homology changes along a filtration. Its usual summaries are barcodes and persistence diagrams. A barcode is a multiset of intervals $[b,d)$, where $b$ is the parameter at which an interval summand appears and $d$ is the parameter at which it disappears. These summaries are stable, computable, and widely used in topological data analysis \cite{edelsbrunner2002topological,zomorodian2005computing, ghrist2008barcodes,edelsbrunnerharer2010computational}.

A barcode answers an important question: \emph{when} does a homological feature live? It does not, however, choose a cellular cycle for each interval, and it does not choose a basis of the homology vector space at each filtration value. Matrix reduction does make such choices \cite{edelsbrunner2002topological, zomorodian2005computing, hang2021umatch}. After fixing a total order of the cells and a deterministic reduction rule, the reduction produces concrete cycles that form a basis of homology at each cellwise stage. This raises a second question: when one of the persistence labels ends, \emph{how is the selected basis rewritten in terms of the labels that remain active?}

The distinction between a persistence interval and a selected basis element is essential. Suppose a positive $q$-cell $\sigma_i$ is paired with a negative $(q+1)$-cell $\sigma_j$. The pair $(i,j)$ determines a persistence interval. The reduced death column $R_j$ represents a cycle that becomes a boundary when $\sigma_j$ enters. By contrast, the transformation column $V_i$ is a selected birth cycle. Its homology class need not become zero when the interval label $i$ ends. It may instead map to a linear combination of selected classes whose labels survive longer. There is no contradiction as the interval generator dies at its endpoint, while a particular reduction-selected cycle carrying the same birth label may continue as a different linear combination.

Dimension zero provides the simplest picture. When two connected components merge, the component whose label is not retained by the elder rule does not vanish as a subset of the later space. It becomes part of the component whose label survives. A merge tree records this hierarchy of component mergers \cite{carr2003computing,morozov2013interleaving, edelsbrunnerharer2010computational}. In higher homological dimensions, there is generally no literal merger of embedded loops or voids but, after inclusion across a filtration value, what are the coordinates of a selected dying label in the basis of selected surviving labels?

This paper answers that question by introducing the \emph{persistence pairing graph}. At a filtration value $t$, let $i_t:K_{t^-}\hookrightarrow K_t$ be the inclusion. The selected classes whose persistence intervals cross $t$ form a basis of $\operatorname{im}(i_t)_*$. Therefore, the image of every selected class whose label dies at $t$ has a unique expansion in that surviving basis. We use the nonzero coefficients of these expansions as directed graph edges. Cohomology gives an equivalent coordinate formula where classes dual to the surviving basis recover the same coefficients through the Kronecker pairing.

The construction is related to merge trees but is not a higher-dimensional merge tree in a geometric sense. It is also not determined by the persistence module alone. The barcode is fixed by the filtered complex and the coefficient field, whereas the graph edges can depend on the cell order used to resolve ties and on the reduced factorization used to select cycle representatives. We therefore treat a persistence pairing graph as a graph summary attached to the full reduction data, not as a replacement for the persistence diagram. This distinction is important both for interpretation and for comparison of graphs.

The main contributions of this study are as follows:
\begin{enumerate}
\item We prove that the selected birth cycles indexed by active positive cells form a basis of homology at every cellwise stage.

\item We define transition coefficients, at each filtration value, by expanding the images of dying selected labels in the survivor basis. We identify the corresponding kernel vectors and recover the same coefficients through the Kronecker pairing with a dual cohomology basis.

\item We define a persistence pairing graph from the support of the transition coefficients, prove that it is a directed acyclic graph, and show that in dimension zero, under the standard reduction convention and a common tie-breaking rule, it agrees with a merge tree.

\item We define a bounded fused Gromov--Wasserstein-type discrepancy that combines interval attributes with directed edge structure \cite{vayer2020fused}. For two filtrations on the same finite complex, we prove a matched-subgraph comparison bound in which the interval term is controlled by ordinary persistence stability \cite{cohensteiner2007stability}, while the edge term remains an explicit reduction-dependent disagreement term.
\end{enumerate}

The paper is organized as follows. Section~\ref{sec:preliminaries} develops the algebraic background and the persistence pairing theorem. Section~\ref{sec:persistence pairing-graphs} defines transition coordinates, a persistence pairing graph, its relation to merge trees, a direct computation procedure, and the fused graph discrepancy. The final section summarizes the interpretation, limitations, and directions for further work.

\section{Preliminaries}
\label{sec:preliminaries}

The basic definitions of CW complexes, cellular chains, and cellular homology
used here are standard; see, for example, \cite{hatcher2002algebraic}.

\subsection{Cellular Homology}
\label{subsec:cellular-homology}

Throughout the paper, $K$ is a finite CW complex and $\Bbbk$ is a fixed field. For each $q\geq 0$, let $\mathcal K_q$ denote the set of $q$-dimensional (open) cells of $K$, and let
\[
\mathcal K=\bigcup_{q\geq 0}\mathcal K_q
\]
be the set of all cells. An orientation is fixed for every cell. When $\operatorname{char}(\Bbbk)=2$, the orientation choices do not affect the boundary coefficients.

A CW complex is built inductively by attaching cells. Starting with the $0$-skeleton $K^0$, an open $n$-cell is attached by a continuous map from the boundary sphere $S^{n-1}=\partial D^n$ of a closed $n$-disk into the $(n-1)$-skeleton. The union of the cells of dimension at most $n$ is the $n$-skeleton. We use cellular chains because a finite CW complex gives finite-dimensional chain groups and therefore finite matrices.

We use the standard filtered-complex setting of persistent homology
\cite{edelsbrunner2002topological,zomorodian2005computing,
edelsbrunnerharer2010computational}.

\begin{definition}
\label{def:filtration-function}
A \emph{filtration function} is a map $f:\mathcal K\longrightarrow\mathbb R$ such that
\begin{equation}
\label{eq:face-monotone}
\tau\subseteq\overline{\sigma} \quad\Longrightarrow\quad f(\tau)\leq f(\sigma)
\end{equation}
for all cells $\sigma,\tau\in\mathcal K$.
\end{definition}

Condition~\eqref{eq:face-monotone} says that every face of a cell enters the filtration no later than the cell itself.

For each threshold $t\in\mathbb R$, the \emph{sublevel subcomplex} is given by
\[
K_t = \bigcup_{\{\sigma\in\mathcal K\,\mid\, f(\sigma)\leq t\}}\sigma.
\]

If $a\leq b$, then $K_a\subseteq K_b$. Since $K$ has finitely many cells, the image $f(\mathcal K) \subseteq \mathbb{R}$ is finite. Write its distinct values in increasing order as
\[
t_0<t_1<\cdots<t_m.
\]
We call these the \emph{filtration values}. The filtration may then be written as
\[
\varnothing\subseteq K_{t_0}\subseteq K_{t_1}\subseteq\cdots\subseteq K_{t_m}=K.
\]
For a filtration value $t$, we define
\[
K_{t^-} := \bigcup_{s<t}K_s = \bigcup_{\{\sigma\in\mathcal K\,\mid\,f(\sigma)<t\}}\sigma.
\]
Hence, for $i \geq 1$, $K_{t_i^-}=K_{t_{i-1}}$ while $K_{t_0^-}=\varnothing$.

For each $q\geq0$, the \emph{cellular $q$-chain group} is the $\Bbbk$-vector space
\[
C_q(K;\Bbbk) := \bigoplus_{\sigma\in\mathcal K_q}\Bbbk\,\sigma
\]
freely generated by the oriented $q$-cells of $K$. Thus, every $q$-chain has a unique expression
\[
c=\sum_{\sigma\in\mathcal K_q}a_\sigma\sigma
\]
given $a_\sigma\in\Bbbk$.

For $q \geq 1$ and $\sigma \in \mathcal K_q$, the \emph{cellular boundary} of $\sigma$ is
\[
\partial_q\sigma = \sum_{\tau\in\mathcal K_{q-1}}[\sigma:\tau]\,\tau,
\]
where $[\sigma:\tau]\in\Bbbk$ is the cellular incidence coefficient of $\sigma$ with respect to $\tau$. It records, with sign determined by the chosen orientations, how the attaching map of $\sigma$ passes over the cell $\tau$. The \emph{boundary operator}
\[
\partial_q:C_q(K;\Bbbk)\longrightarrow C_{q-1}(K;\Bbbk).
\]
is obtained by extending this assignment linearly. We set $\partial_0 = 0$. The cellular boundary operators satisfy
\begin{equation}
\label{eq:boundary-square-zero}
\partial_{q-1}\partial_q=0.
\end{equation}

The \emph{spaces of $q$-cycles} and \emph{$q$-boundaries} are
\[
Z_q(K;\Bbbk):=\ker\partial_q
\qquad\text{and}\qquad
B_q(K;\Bbbk):=\operatorname{im}\partial_{q+1}.
\]
Equation~\eqref{eq:boundary-square-zero} implies $B_q(K;\Bbbk)\subseteq Z_q(K;\Bbbk)$. The \emph{$q$-th homology group} is
\[
H_q(K;\Bbbk) := Z_q(K;\Bbbk)/B_q(K;\Bbbk).
\]
Two cycles $z$ and $z'$ represent the same homology class exactly when $z-z'$ is a boundary. Cellular homology is naturally isomorphic to singular homology of the underlying topological space of K \cite{hatcher2002algebraic}. Thus the cellular chain complex provides a finite, cell-based computation of the same homology groups.

\subsection{Cellular Cohomology}
\label{subsec:cohomology}

The \emph{cellular $q$-cochain group} is the linear dual
\[
C^q(K;\Bbbk) := \operatorname{Hom}_{\Bbbk}\bigl(C_q(K;\Bbbk),\Bbbk\bigr).
\]
of the cellular chain group. Thus, a $q$-cochain is a $\Bbbk$-linear functional on $q$-chains and is determined by its values on the $q$-cells. The \emph{coboundary operator} $\delta^q:C^q(K;\Bbbk)\longrightarrow C^{q+1}(K;\Bbbk)$ is the dual of $\partial_{q+1}$ defined by
\[
\delta^q\varphi = \varphi\circ\partial_{q+1}.
\]
Thus, for $c\in C_{q+1}(K;\Bbbk)$,
\[
(\delta^q\varphi)(c)=\varphi(\partial_{q+1}c).
\]
Since $\partial_{q+1}\partial_{q+2} = 0$, the coboundary operators satisfy
\[
\delta^{q+1}\delta^q=0.
\]
The \emph{spaces of cocycles} and \emph{coboundaries} are $Z^q(K;\Bbbk):=\ker\delta^q$ and, for $q \geq 1$, $B^q(K;\Bbbk):=\operatorname{im}\delta^{q-1}$, with $B^0(K;\Bbbk) := 0$. Since $\delta^q\delta^{q-1} = 0$, we have $B^q(K;\Bbbk) \subseteq Z^q(K;\Bbbk)$. The \emph{$q$-th cellular cohomology group} is
\[
H^q(K;\Bbbk) := Z^q(K;\Bbbk)/B^q(K;\Bbbk).
\]
Two cocyles represent the same cohomology class exactly when their difference is a coboundary.

There is a natural evaluation pairing $C^q(K;\Bbbk)\times C_q(K;\Bbbk)\longrightarrow\Bbbk$ defined by 
\[
(\varphi,c)\longmapsto\varphi(c).
\]
To show that evaluation induces a well-defined pairing on homology and cohomology, we verify that it is independent of the chosen representatives. Let $\varphi \in Z^q(K;\Bbbk)$ and $c \in Z_q(K;\Bbbk)$. If $c$ is replaced by the homologous cycle $c+\partial_{q+1}b$, where $b \in C_{q+1}(K;\Bbbk)$, then
\[
\varphi(c+\partial_{q+1}b) = \varphi(c)+(\delta^q\varphi)(b) = \varphi(c).
\]
Similarly, if $\varphi$ is replaced by the cohomologous cocycle $\varphi+\delta^{q-1}\eta$, where $\eta \in C^{q-1}(K;\Bbbk)$, then
\[
(\varphi+\delta^{q-1}\eta)(c) = \varphi(c)+\eta(\partial_qc) = \varphi(c).
\]
Hence, evaluation depends only on the classes $[\varphi]$ and $[c]$.

\begin{definition}
\label{def:kronecker-pairing}
The \emph{Kronecker pairing} is the map
\[
\langle\cdot,\cdot\rangle_K:
H^q(K;\Bbbk)\times H_q(K;\Bbbk)
\longrightarrow\Bbbk
\]
defined by $\langle[\varphi],[c]\rangle_K:=\varphi(c)$.
\end{definition}

Because $\Bbbk$ is a field, the universal coefficient theorem gives a natural
isomorphism \cite{hatcher2002algebraic}
\[
H^q(K;\Bbbk) \cong \operatorname{Hom}_{\Bbbk}\bigl(H_q(K;\Bbbk),\Bbbk\bigr).
\]
Thus, the Kronecker pairing is nondegenerate in both arguments. This means that every nonzero homology class pairs nontrivially with some cohomology class, and every nonzero cohomology class pairs nontrivially with some homology class. 

Since $K$ is finite, $H_q(K;\Bbbk)$ is finite-dimensional. Hence, every basis $([z_1],\ldots,[z_r])$ of $H_q(K;\Bbbk)$ determines a unique dual basis $([\alpha_1],\ldots,[\alpha_r])$ of $H^q(K;\Bbbk)$ satisfying
\[
\langle[\alpha_i],[z_j]\rangle_K=\delta_{ij},
\]
where
\[
\delta_{ij} =
\begin{cases}
1,& i=j,\\
0,& i\neq j
\end{cases}
\]
is the \emph{Kronecker delta}.

If $L\subseteq K$ is a subcomplex and $i:L\hookrightarrow K$ is the inclusion, then $i$ induces
\[
i_*:H_q(L;\Bbbk)\longrightarrow H_q(K;\Bbbk)
\]
and, in the opposite direction,
\[
i^*:H^q(K;\Bbbk)\longrightarrow H^q(L;\Bbbk).
\]
The following identity is a property of the Kronecker pairing
\cite{hatcher2002algebraic}. We include the short proof because the identity
will be used repeatedly.

\begin{lemma}
\label{lem:kronecker-naturality}
For every $\alpha\in H^q(K;\Bbbk)$ and $x\in H_q(L;\Bbbk)$,
\[
\langle i^*\alpha,x\rangle_L
=
\langle\alpha,i_*x\rangle_K.
\]
\end{lemma}

\begin{proof}
Choose a cocycle $\varphi\in C^q(K;\Bbbk)$ representing $\alpha$ and a cycle $z\in C_q(L;\Bbbk)$ representing $x$. Let
\[
i_\#:C_q(L;\Bbbk)\longrightarrow C_q(K;\Bbbk)
\]
be the cellular chain map induced by inclusion. The corresponding cochain map is 
\[
i^\#:C^q(K;\Bbbk)\longrightarrow C^q(L;\Bbbk),
\]
where $i^\#\varphi=\varphi\circ i_\#$. Hence, $i^*[\varphi] = [i^\#\varphi]$ and $i_*[z] = [i_\# z]$. Therefore,
\[
\begin{aligned}
\langle i^*\alpha,x\rangle_L
&=\langle[i^\#\varphi],[z]\rangle_L\\
&=(i^\#\varphi)(z)\\
&=\varphi(i_\#z)\\
&=\langle[\varphi],[i_\#z]\rangle_K\\
&=\langle\alpha,i_*x\rangle_K.
\end{aligned}
\]
\end{proof}

\subsection{Persistence Module} 
\label{subsec:persistence-modules}

For $a\leq b$, the inclusion $i_{a,b}:K_a\hookrightarrow K_b$ induces a linear map
\[
(i_{a,b})_*: H_q(K_a;\Bbbk) \longrightarrow H_q(K_b;\Bbbk).
\]
These maps satisfy
\[
(i_{a,a})_*=\operatorname{id}_{H_q(K_a;\Bbbk)}
\]
and, for $a\leq b\leq c$,
\[
(i_{b,c})_*\circ(i_{a,b})_*=(i_{a,c})_*
\]

Let $\mathbb V_t^{(q)} := H_q(K_t;\Bbbk)$ and $v_{a,b}^{(q)} := (i_{a,b})_*$. The collection of vector spaces $\mathbb V_t^{(q)}$ together with the linear maps
\[
v_{a,b}^{(q)}: \mathbb V_a^{(q)} \longrightarrow \mathbb V_b^{(q)} \qquad(a\leq b)
\]
is the \emph{$q$-th persistence module} associated with the filtration. The structure maps satisfy $v_{a,a}^{(q)} = \operatorname{id}_{\mathbb V_a^{(q)}}$ and, for $a\leq b\leq c$,
\[
v_{b,c}^{(q)} \circ v_{a,b}^{(q)} = v_{a,c}^{(q)}.
\]

For $b<d\leq\infty$, the \emph{interval persistence module} $I[b,d)$ is given by
\[
I[b,d)_t =
\begin{cases}
\Bbbk,&b\leq t<d,\\
0,&\text{otherwise}.
\end{cases}
\]
For $a \leq c$, the map $I[b, d)_a \to I[b, d)_c$ is the identity on $\Bbbk$ when $b \leq a \leq c < d$, and the zero map otherwise. 

Since $K$ is finite, the persistence module $\mathbb V^{(q)}$ is pointwise finite-dimensional and changes at only finitely many filtration values. The interval decomposition theorem therefore yields an isomorphism
\begin{equation}
\label{eq:interval-decomposition}
\mathbb V^{(q)} \cong \bigoplus_{u\in\mathcal V^{(q)}}I[b_u,d_u)
\end{equation}
for a finite index set $\mathcal V^{(q)}$ \cite{zomorodian2005computing,crawleyboevey2015decomposition}. The interval $[b_u, d_u)$ records the lifetime of the corresponding persistence class. The expression $b_u$ is its \emph{birth} and $d_u$ is its \emph{death}, with $d_u = \infty$ for a class that survives to the final complex.

The multiset of intervals
\[
\operatorname{Bar}_q(K,f) := \bigl\{\!\bigl\{[b_u,d_u)\mid u\in\mathcal V^{(q)}\bigr\}\!\bigr\}
\]
is called the \emph{$q$-th barcode} of the filtered complex $(K,f)$. Each finite interval $[b_u,d_u)$ corresponds to the point $(b_u,d_u)$ in the $q$-th persistence diagram. The \emph{persistence} of the interval is the difference
\[
d_u-b_u.
\]
An interval of the form $[b_u,\infty)$ is called \emph{essential}. It represents a homology class that survives to the final complex and is assigned infinite persistence.

The interval decomposition theorem determines the multiset of intervals uniquely, including their multiplicities, but the particular isomorphism in~\eqref{eq:interval-decomposition} is generally not canonical. Consequently, the barcode does not determine a preferred cellular cycle representing each interval. Different reduction procedures, or different choices made within a reduction, may produce different cycle representatives while yielding the same barcode \cite{hang2021umatch}.

For $a\leq b$, the inclusion $i_{a,b}$ also induces a contravariant map on cohomology,
\[
i_{a,b}^*: H^q(K_b;\Bbbk) \longrightarrow H^q(K_a;\Bbbk).
\]
Under the natural identifications
\[
H^q(K_t;\Bbbk) \cong \operatorname{Hom}_{\Bbbk}\bigl(H_q(K_t;\Bbbk),\Bbbk\bigr),
\]
Lemma~\ref{lem:kronecker-naturality} identifies $i_{a,b}^*$ with the linear dual of
\[
(i_{a,b})_*: H_q(K_a;\Bbbk) \longrightarrow H_q(K_b;\Bbbk).
\]
Thus, after accounting for the reversal of the indexing direction, the persistent homology module and the corresponding persistent cohomology system determine the same barcode information \cite{desilva2011dualities}.

\subsection{Persistence Pairs}
\label{subsec:filtered-matrix-reduction}

We now recall the standard matrix-reduction formulation of persistent homology \cite{edelsbrunner2002topological,zomorodian2005computing, edelsbrunnerharer2010computational}. The notation below is chosen so that the change-of-basis matrix is retained, because its columns will later supply the selected cycle representatives.

To represent the filtered cellular chain complex by a matrix, we first choose a total order
\[
\sigma_1\prec\sigma_2\prec\cdots\prec\sigma_N
\]
on the cells of $K$. We require this order to be compatible both with the filtration and with the cellular boundary structure. That is,
\[
f(\sigma_i)<f(\sigma_j)\quad\Longrightarrow\quad i<j
\]
and
\[
\sigma_i\subseteq\overline{\sigma_j},\ \sigma_i\neq\sigma_j \quad\Longrightarrow\quad i<j.
\]
The first condition ensures that a cell with a smaller filtration value is placed earlier in the order. The second ensures that every proper cell contained in the closure of another cell is placed before that cell. These requirements are compatible because the filtration function satisfies
\[
\sigma_i\subseteq\overline{\sigma_j} \quad\Longrightarrow\quad f(\sigma_i)\leq f(\sigma_j).
\]
If two cells have the same filtration value and neither condition determines their relative position, the total order $\prec$ provides an arbitrary tie-breaking order.

The chosen order refines the original filtration to a \emph{cellwise filtration} $K^0:=\varnothing$ and
\[
K^p:=\bigcup_{\ell=1}^{p}\sigma_\ell
\]
for $1\leq p\leq N$. Thus,
\[
\varnothing = K^0 \subseteq K^1 \subseteq \cdots \subseteq K^N = K,
\]
and exactly one cell is added at each step. Every $K^p$ is a CW subcomplex. This means that if $\sigma_\ell \subseteq \overline{\sigma_p}$ with $\sigma_\ell\neq\sigma_p$, then $\ell<p$, so all cells required by the boundary structure of $\sigma_p$ are already contained in $K^{p-1}$.

The cellwise filtration is a refinement of the original filtration rather than a new filtration function. In particular, if several cells have a common filtration value $t$, the original filtration adds all of them simultaneously in passing from
$K_{t^-}$ to $K_t$. The corresponding intermediate complexes in the cellwise filtration depend on the chosen tie-breaking order and are auxiliary stages used for the matrix reduction. They need not correspond to distinct parameter values of the original filtration.

Let
\[
C_*(K;\Bbbk) := \bigoplus_{q\geq0} C_q(K;\Bbbk)
\]
be the total cellular chain space. The cellular boundary operators combine to give the linear map
\[
\partial:C_*(K;\Bbbk)\longrightarrow C_*(K;\Bbbk),
\]
whose restriction to $C_q(K;\Bbbk)$ is $\partial_q$. Using the ordered cells $(\sigma_1,\ldots,\sigma_N)$ as a basis of $C_*(K;\Bbbk)$, we identify cellular chains with coordinate column vectors. The matrix of $\partial$ in this basis is the \emph{boundary matrix} $D\in\Bbbk^{N\times N}$ with entries
\[
D_{ij} = [\sigma_j:\sigma_i].
\]
Equivalently, the $j$-th column of $D$ is the coordinate vector of the cellular boundary
\[
\partial\sigma_j = \sum_{i=1}^{N}D_{ij}\sigma_i.
\]

If $D_{ij}\neq0$, then $\sigma_i$ occurs with nonzero incidence coefficient in the boundary of $\sigma_j$. In particular,
$\sigma_i\subseteq\overline{\sigma_j}$ whenever $\sigma_i\neq\sigma_j$, and therefore the ordering condition gives
\[
i<j.
\]
Hence, every nonzero entry of $D$ lies strictly above the diagonal, so $D$ is strictly upper triangular. Moreover, the cellular identity $\partial^2=0$ is represented algebraically by $D^2=0$.

For a nonzero column vector $c\in\Bbbk^N$, we define its \emph{low index} by
\[
\operatorname{low}(c) := \max\{i\mid c_i\neq0\}.
\]
A matrix is \emph{reduced} if no two nonzero columns have the same low index.

Starting from the boundary matrix $D$, the standard left-to-right persistence algorithm applies elementary column operations until a reduced matrix $R$ is obtained. Recording the same column operations in a change of basis matrix $V$ gives
\begin{equation}
\label{eq:RDV}
R=DV,
\end{equation}
where $V$ is upper unitriangular, that is, upper triangular with $V_{ii}=1$ for $1\leq i\leq N$. We perform the reduction per dimension, so each column $V_i$ is supported only on cells having the same dimension as $\sigma_i$.

The left-to-right reduction, the low index pairing rule, and the resulting birth--death pairs are the standard boundary matrix reduction \cite{edelsbrunner2002topological,zomorodian2005computing}. Retaining transformation matrices also permits explicit recovery of cycle representatives; see, for example, \cite{hang2021umatch}.

Since $V$ is upper unitriangular, its $i$-th column has the form
\[
V_i = e_i+\sum_{\ell<i}V_{\ell i}e_\ell,
\]
where $e_i$ denotes the $i$-th coordinate vector. Thus, $V_i$ represents a chain containing $\sigma_i$ with coefficient $1$, together with a linear combination of earlier cells of the same dimension.

\begin{definition}
\label{def:positive-cells}
A cell $\sigma_i$ is called \emph{positive} if the $i$-th reduced column satisfies $R_i=0$, and \emph{negative} if $R_i\neq0$. If $\sigma_i$ is a positive $q$-cell, its \emph{selected birth cycle} is the $q$-chain
\begin{equation}
\label{eq:selected-birth-cycle}
z_i := \sum_{\substack{\ell\leq i\\dim\sigma_\ell=q}}V_{\ell i}\sigma_\ell.
\end{equation}
\end{definition}

Equivalently, in Definition~\ref{def:positive-cells},
\[
z_i = \sigma_i+ \sum_{\substack{\ell<i\\dim\sigma_\ell=q}}V_{\ell i}\sigma_\ell.
\]
Its coordinate column in the ordered cellular basis is $V_i$.

\begin{lemma}
\label{lem:selected-representative-cycle}
If $\sigma_i$ is a positive $q$-cell, then $z_i \in Z_q(K^i;\Bbbk)$.
\end{lemma}

\begin{proof}
Since $V$ is upper triangular, $V_{\ell i}=0$ whenever $\ell>i$, and dimension preservation ensures that $V_i$ is supported only on $q$-cells. Hence, $z_i\in C_q(K^i;\Bbbk)$. The $i$-th column of~\eqref{eq:RDV} is $R_i=DV_i$. Since $D$ represents the cellular boundary operator and $V_i$ is the coordinate vector of $z_i$,
\[
\partial_q z_i=DV_i=R_i.
\]
Because $\sigma_i$ is positive, $R_i=0$. Therefore, $\partial_qz_i=0$, and hence $z_i\in Z_q(K^i;\Bbbk)$.
\end{proof}

For the succeeding discussions, we will use the following elementary facts repeatedly.

\begin{lemma}
\label{lem:reduced-matrix-facts}
Let $U\in\Bbbk^{N\times N}$ be upper unitriangular.
\begin{enumerate}
\item If $c\in\Bbbk^N$ is nonzero, then $\operatorname{low}(Uc)=\operatorname{low}(c)$.
\item The nonzero columns of a reduced matrix are linearly independent.
\end{enumerate}
\end{lemma}

\begin{proof}
Let $i=\operatorname{low}(c)$. Then $c_i\neq0$ and $c_\ell=0$ for every $\ell>i$. Since $U$ is upper triangular, $Uc$ has no nonzero entries below row $i$. Since $U_{ii}=1$, its $i$-th entry is $c_i\neq0$. Hence,
\[
\operatorname{low}(Uc)=i.
\]

For the second statement, suppose, on the contraty, that a nontrivial linear combination of nonzero reduced columns is zero. Thus, for some nonempty index set $J$,
\[
\sum_{j\in J} a_j R_j = 0,
\]
where $a_j\neq 0$ for all $j\in J$. Choose $j_*\in J$ such that
\[
i:=\operatorname{low}(R_{j_*}) = \max_{j\in J}\operatorname{low}(R_j).
\]
Since $R$ is reduced, no other column $R_j$ with $j\in J$ has low index $i$. Hence, for every $j\in J\setminus\{j_*\}$,
\[
\operatorname{low}(R_j)<i,
\]
so $(R_j)_i=0$. Taking the $i$-th coordinate of the relation therefore gives
\[
0 = \sum_{j\in J} a_j(R_j)_i = a_{j_*}(R_{j_*})_i.
\]
But $a_{j_*}\neq0$ and $(R_{j_*})_i\neq0$ by the definition of $\operatorname{low}(R_{j_*})$. Since $\Bbbk$ is a field, their product is nonzero, which is a contradiction. Therefore, the nonzero columns of a reduced matrix are linearly independent.
\end{proof}

\begin{definition}
\label{def:reduction-pair}
Suppose $R_j\neq0$ and $\operatorname{low}(R_j)=i$. Then $(i,j)$ is called a \emph{persistence pair}, and we say that the cell $\sigma_i$ is \emph{paired} with $\sigma_j$. If $\sigma_i$ is positive and there is no column $R_j$ with $\operatorname{low}(R_j)=i$, then $\sigma_i$ is called \emph{unpaired}.
\end{definition}

Since $R$ is reduced, at most one such index $j$ can exist. For every positive cell $\sigma_i$, let
\[
\operatorname{pair}(i) :=
\begin{cases}
j,&\operatorname{low}(R_j)=i,\\
\infty,&\sigma_i\text{ is unpaired}.
\end{cases}
\]

A standard property of persistence reduction is that a nonzero reduced column pairs a positive cell with a cell one dimension higher \cite{edelsbrunner2002topological,zomorodian2005computing}. We give a proof in the present notation because the result is used repeatedly.

\begin{lemma}
\label{lem:pivot-positive}
If $\operatorname{low}(R_j)=i$, then $R_i=0$ and $\dim\sigma_j=\dim\sigma_i+1$.
\end{lemma}

\begin{proof}
Set $\lambda:=V^{-1}R_j$. Since $V^{-1}$ is upper unitriangular, Lemma~\ref{lem:reduced-matrix-facts} gives
\[
\operatorname{low}(\lambda) = \operatorname{low}(R_j) = i.
\]
Hence, $\lambda_i\neq 0$ and $\lambda_\ell=0$ for all $\ell>i$. Since $D^2=0$, we have $DR_j=0$. Using $R=DV$ and $R_j=V\lambda$, we obtain
\[
0 =DR_j =DV\lambda =R\lambda =\sum_{\ell\leq i}\lambda_\ell R_\ell.
\]
If $R_i\neq 0$, then the term $\lambda_iR_i$ is nonzero. By Lemma~\ref{lem:reduced-matrix-facts}, the nonzero columns of $R$ are linearly independent, so this term cannot be cancelled by the remaining terms. Therefore, $R_i=0$. Thus $\sigma_i$ is positive.

Finally, $V_j$ is supported only on cells having the same dimension as $\sigma_j$. Since $R_j=DV_j$ is its cellular boundary, $R_j$ is supported only on cells one dimension lower. Because row $i$ occurs with nonzero coefficient in $R_j$,
\[
\dim\sigma_i = \dim\sigma_j-1.
\]
Equivalently, $\dim\sigma_j = \dim\sigma_i+1$.
\end{proof}

For a positive $q$-cell $\sigma_i$ and any $p\geq i$, the selected birth cycle $z_i$ remains a $q$-cycle in $K^p$. We denote its homology class by
\[
[z_i]_p \in H_q(K^p;\Bbbk).
\]
Notice that enlarging the complex does not destroy the cycle condition $\partial_q z_i=0$. A class dies only when its representing cycle becomes a boundary in a later complex.

\begin{definition}
\label{def:active-positive-index}
A positive $q$-cell index $i$ is \emph{active at stage $p$} if $i\leq p<\operatorname{pair}(i)$. We denote the set of active positive $q$-cell indices at stage $p$ by
\[
\mathcal A_q(p) := \left\{i\ \middle|\ \dim\sigma_i=q,\; R_i=0,\; i\leq p<\operatorname{pair}(i)\right\}.
\]
\end{definition}

The change-of-basis matrix from the boundary matrix reduction can be used to recover cycle representatives and bases of cycle and boundary spaces \cite{edelsbrunnerharer2010computational,hang2021umatch}. The following theorem gives the precise stagewise basis statement needed for a persistence pairing graph.

\begin{theorem}
\label{thm:active-basis}
For every stage $p$ and every dimension $q$,
\[
\mathcal B_q(p) := \{[z_i]_p\mid i\in\mathcal A_q(p)\}
\]
is a basis of $H_q(K^p;\Bbbk)$.
\end{theorem}

\begin{proof}
Consider the columns $V_i$ satisfying $i\leq p$ and $\dim\sigma_i=q$. Because $V$ is upper unitriangular and preserves chain dimension, these columns form a basis of $C_q(K^p;\Bbbk)$. Hence, every $q$-chain $c\in C_q(K^p;\Bbbk)$ can be written uniquely as
\[
c = \sum_{\substack{i\leq p\\ \dim\sigma_i=q}}a_iV_i.
\]
Applying the boundary matrix gives
\[
Dc = \sum_{\substack{i\leq p\\ \dim\sigma_i=q}}a_iDV_i = \sum_{\substack{i\leq p\\ \dim\sigma_i=q}} a_iR_i.
\]
The nonzero columns of $R$ are linearly independent. Therefore, $Dc=0$ if and only if $a_i=0$ for every $i$ with $R_i\neq 0$. Consequently,
\[
\mathcal{Z} := \left\{z_i\ \mid \ i \leq p,\;\dim\sigma_i=q, \text{ and }R_i=0\right\}
\]
is a basis of $Z_q(K^p;\Bbbk)$.

The columns $V_j$ satisfying $j\leq p$ and $\dim\sigma_j=q+1$ form a basis of $C_{q+1}(K^p;\Bbbk)$. Their boundaries are
\[
DV_j=R_j.
\]
Hence, the nonzero columns
\[
\left\{R_j\ \mid j\leq p,\;\dim\sigma_j=q+1,\text{ and } R_j\neq 0\right\}
\]
span $B_q(K^p;\Bbbk)$. They are linearly independent because $R$ is reduced, so they form a basis of the boundary space.

Let $R_j\neq 0$ with $j\leq p$ and $\dim\sigma_j=q+1$, and set $i:=\operatorname{low}(R_j)$. Since $R_j$ is a $q$-boundary, it is also a $q$-cycle. Hence, it has a unique expansion in the cycle basis $\mathcal{Z}$.

Let $\lambda:=V^{-1}R_j$. By Lemma~\ref{lem:reduced-matrix-facts},
\[
\operatorname{low}(\lambda) = \operatorname{low}(R_j) = i.
\]
Therefore, $\lambda_i\neq 0$ and $\lambda_k=0$ for all $k>i$.

Since only positive cycle basis vectors can occur in the expansion of a cycle, we may write
\[
R_j = \lambda_i z_i + \sum_{\substack{k<i\\ R_k=0}}\lambda_k z_k,
\]
where $\lambda_i\neq 0$. By Lemma~\ref{lem:pivot-positive}, the leading index $i$ is positive.

Distinct nonzero columns $R_j$ have distinct low indices because $R$ is reduced. Thus, the boundary relations have distinct leading positive indices. These leading indices are precisely the positive indices $i$ satisfying $\operatorname{pair}(i)\leq p$.

For such an index $i$, let $j=\operatorname{pair}(i)$. Since
\[
R_j = \lambda_i z_i + \sum_{\substack{k<i\\R_k=0}} \lambda_k z_k
\]
given $\lambda_i \neq 0$ and $R_j$ is a boundary in $K^p$, its homology class is zero. Hence,
\[
0 = \lambda_i[z_i]_p + \sum_{\substack{k<i\\R_k=0}}\lambda_k[z_k]_p.
\]
Since $\lambda_i\neq0$, we obtain
\[
[z_i]_p = -\lambda_i^{-1}\sum_{\substack{k<i\\R_k=0}}\lambda_k[z_k]_p.
\]
Thus, whenever the paired death cell of $i$ has appeared, the class $[z_i]_p$ is no longer an independent homology class, that is, it is determined by classes with smaller indices.

The remaining basis classes are precisely those satisfying $i\leq p<\operatorname{pair}(i)$, which are exactly those with $i\in\mathcal A_q(p)$. Hence,
\[
\mathcal B_q(p) = \{[z_i]_p\mid i\in\mathcal A_q(p)\}
\]
is a basis of $H_q(K^p;\Bbbk)$.
\end{proof}

Theorem~\ref{thm:active-basis} describes a homology basis selected by the change-of-basis matrix $V$. A negative reduced column determines a different, but closely related, cycle. This cycle is the homology class that is actually killed when the corresponding negative cell enters the cellwise filtration.

Reduced columns also encode cycle representatives associated with persistence pairs; see \cite{edelsbrunnerharer2010computational,hang2021umatch}. The next proposition records the precise death cycle property needed below.

\begin{proposition}
\label{prop:killed-cycle}
Suppose $\operatorname{low}(R_j)=i$ and $\dim\sigma_i=q$. Then $R_j$, as a cellular chain, is a $q$-cycle supported in $K^i$. Its homology class is nonzero in $H_q(K^{j-1};\Bbbk)$, becomes zero in $H_q(K^j;\Bbbk)$, and
\[
\ker\!\left(H_q(K^{j-1};\Bbbk) \longrightarrow H_q(K^j;\Bbbk)\right) = \operatorname{span}_{\Bbbk}\{[R_j]_{j-1}\}.
\]
\end{proposition}

\begin{proof}
By Lemma~\ref{lem:pivot-positive}, $\dim\sigma_j=q+1$. Since the reduction preserves chain dimension, the column $V_j$ represents a $(q+1)$-chain. Therefore, $R_j=DV_j$ is a $q$-chain.

Because $\operatorname{low}(R_j)=i$, all nonzero entries of $R_j$ occur in rows with index at most $i$. Hence, $R_j$ is supported on $q$-cells contained in $K^i$, so $R_j\in C_q(K^i;\Bbbk)$. Moreover,
\[
DR_j = D(DV_j) = D^2V_j = 0.
\]
Thus, $R_j$ is a $q$-cycle.

Since $V$ is upper unitriangular and preserves dimension, the column $V_j$ has the form
\[
V_j = e_j + \sum_{\substack{\ell<j\\ \dim\sigma_\ell=q+1}} V_{\ell j}e_\ell.
\]
Hence, $V_j$ represents a $(q+1)$-chain in $K^j$. Therefore, $R_j=DV_j$ is a boundary in $K^j$, and consequently $[R_j]_j=0$.

We next show that $R_j$ is not already a boundary in $K^{j-1}$. By the proof of Theorem~\ref{thm:active-basis},
\[
B_q(K^{j-1};\Bbbk) = \operatorname{span}_{\Bbbk}\left\{R_\ell \ \middle|\ \ell<j,\; \dim\sigma_\ell=q+1, \text{ and } R_\ell\neq 0\right\},
\]
and these nonzero reduced columns are linearly independent. If $R_j$ belonged to this boundary space, then $R_j$ would be a linear combination of earlier nonzero reduced columns. This would contradict the linear independence of the nonzero columns of the reduced matrix. Hence, $R_j\notin B_q(K^{j-1};\Bbbk)$, and therefore $[R_j]_{j-1}\neq0$.

Passing from $K^{j-1}$ to $K^j$ adds only the $(q+1)$-cell $\sigma_j$. Consequently, the $q$-chain group and the $q$-cycle space do not change. That is, 
\[
C_q(K^{j-1};\Bbbk) = C_q(K^j;\Bbbk) \quad\text{and}\quad Z_q(K^{j-1};\Bbbk) = Z_q(K^j;\Bbbk).
\]

On the other hand, the $q$-boundary space acquires one new independent boundary. Indeed, since
\[
V_j = e_j + \sum_{\substack{\ell<j\\ \dim\sigma_\ell=q+1}} V_{\ell j}e_\ell,
\]
we have
\[
R_j = DV_j = D_j + \sum_{\substack{\ell<j\\ \dim\sigma_\ell=q+1}} V_{\ell j}D_\ell.
\]
The second term is already contained in $B_q(K^{j-1};\Bbbk)$. Thus, adjoining the original boundary $D_j$ or the reduced boundary $R_j$ produces the same enlarged boundary space. Since $R_j\notin B_q(K^{j-1};\Bbbk)$, we obtain the direct sum decomposition
\[
B_q(K^j;\Bbbk) = B_q(K^{j-1};\Bbbk) \oplus \operatorname{span}_{\Bbbk}\{R_j\}.
\]

Because the cycle space is unchanged, the inclusion induced map $H_q(K^{j-1};\Bbbk) \longrightarrow H_q(K^j;\Bbbk)$ is obtained by modding out the same cycle space by the enlarged boundary space. Its kernel is therefore precisely the one-dimensional subspace generated by the class of the newly added boundary. In other words,
\[
\ker\!\left(H_q(K^{j-1};\Bbbk) \longrightarrow H_q(K^j;\Bbbk) \right) = \operatorname{span}_{\Bbbk}\{[R_j]_{j-1}\}.
\]
\end{proof}

For a positive index $i$, it is useful to distinguish the selected birth cycle $z_i$ from a cycle that represents the persistence interval associated with $i$ throughout its full lifetime. 

\begin{definition}
\label{def:interval-cycle}
Let $\sigma_i$ be a positive $q$-cell. The \emph{interval cycle} is
\begin{equation}
\label{eq:interval-cycle}
\rho_i :=
\begin{cases}
R_j, & \text{if }\operatorname{pair}(i)=j<\infty,\\
z_i, & \text{if }\operatorname{pair}(i)=\infty.
\end{cases}
\end{equation}
\end{definition}

In an interval cycle, each reduced column $R_j$ is identified with the cellular chain whose coordinate vector is $R_j$.

For a finite persistence pair $(i,j)$, the choice $\rho_i=R_j$ is different from the selected birth cycle $z_i$. The following result explains why $R_j$ represents the class born at $i$.

\begin{proposition}
\label{prop:interval-cycle-birth}
Let $\sigma_i$ be a positive $q$-cell with $\operatorname{pair}(i)=j<\infty$. Then $\rho_i=R_j$ is a $q$-cycle contained in $K^i$, and its homology class is born at stage $i$.
\end{proposition}

\begin{proof}
Since $R_j=DV_j$, the identity $D^2=0$ gives $DR_j=0$. Hence, $R_j$ is a cycle. Moreover, $\operatorname{low}(R_j)=i$, so every nonzero entry of $R_j$ occurs in a row with index at most $i$. Therefore $R_j$ is supported in $K^i$.

It remains to show that its homology class is born precisely at stage $i$. Set $\lambda:=V^{-1}R_j$. Since $V^{-1}$ is upper unitriangular,
\[
\operatorname{low}(\lambda) = \operatorname{low}(R_j) = i.
\]
Consequently, $\lambda_i\neq 0$ and $\lambda_\ell=0$ for all $\ell>i$.

Because $R_j$ is a cycle, its expansion in the $V$ basis contains only positive columns. Thus, as a cellular chain,
\[
R_j = \lambda_i z_i + \sum_{\substack{\ell<i\\ R_\ell=0}}\lambda_\ell z_\ell,
\]
where $\lambda_i\neq 0$.

The cell $\sigma_i$ is a positive $q$-cell. Passing from $K^{i-1}$ to $K^i$ therefore adds a $q$-cell but no $(q+1)$-cell, so no new $q$-boundaries are created at this step. By Theorem~\ref{thm:active-basis}, the image of $H_q(K^{i-1};\Bbbk) \longrightarrow H_q(K^i;\Bbbk)$ is spanned by the active selected classes whose indices are strictly smaller than $i$, whereas $[z_i]_i$ is the new basis class created at stage $i$.

The coefficient of $z_i$ in the expansion of $R_j$ is $\lambda_i\neq 0$. Boundary relations involving earlier positive indices have leading index strictly smaller than $i$, so passing to homology cannot eliminate this nonzero $[z_i]_i$ component. Hence, $[R_j]_i$ does not lie in the image of $H_q(K^{i-1};\Bbbk)$. Therefore, the class represented by $R_j$ is born at stage $i$.
\end{proof}

The correspondence between reduced matrix pairs and persistence intervals is the classical persistence pairing result \cite{edelsbrunner2002topological,zomorodian2005computing, edelsbrunnerharer2010computational}. We state it in the present notation and give a proof that also keeps track of the interval cycles introduced above.

\begin{theorem}
\label{thm:persistence pairing}
For the cellwise filtration
\[
K^0 \subseteq K^1 \subseteq \cdots \subseteq K^N,
\]
every persistence pair $(i,j)$ determines the interval
\[
[i,j),
\]
and every unpaired positive index $i$ determines the interval
\[
[i,\infty).
\]
For every positive index $i$, the interval cycle $\rho_i$ defined in~\eqref{eq:interval-cycle} represents the corresponding homology class throughout exactly this lifetime.

For the original filtration determined by $f$, a persistence pair $(i,j)$ determines the interval
\[
[f(\sigma_i),f(\sigma_j)),
\]
whereas an unpaired positive cell $\sigma_i$ determines
\[
[f(\sigma_i),\infty).
\]
If $f(\sigma_i)=f(\sigma_j)$, then the interval $[i,j)$ is created entirely by the auxiliary tie-breaking refinement and has zero length on the original filtration value scale. Such an interval is omitted from the usual
barcode.

\end{theorem}

\begin{proof}
Suppose first that $i$ is paired and $\operatorname{pair}(i)=j$. By Proposition~\ref{prop:interval-cycle-birth}, the cycle
\[
\rho_i=R_j
\]
represents a class born at stage $i$.

By Proposition~\ref{prop:killed-cycle}, $[R_j]_{j-1}\neq 0$ and $[R_j]_j=0$. We claim that $[R_j]_p\neq0$ for every $i\leq p<j$. Suppose, to the contrary, that $[R_j]_p=0$ for some such $p$. The inclusion $K^p\hookrightarrow K^{j-1}$ carries the cellular cycle $R_j$ to the same cellular cycle $R_j$. Hence, the induced map on homology would send $[R_j]_p=0$ to $[R_j]_{j-1}=0$. This contradicts Proposition~\ref{prop:killed-cycle}. Hence, $[R_j]_p\neq0$ for every $i\leq p<j$. Since $R_j$ becomes a boundary when $\sigma_j$ enters, its homology class becomes zero at stage $j$. Therefore, its lifetime is $[i,j)$.

Now suppose that $i$ is an unpaired positive index. In this case, $\rho_i=z_i$. By Lemma~\ref{lem:selected-representative-cycle}, $z_i$ is a cycle. At stage $i$, Theorem~\ref{thm:active-basis} shows that $[z_i]_i$ is a new basis element of $H_q(K^i;\Bbbk)$, so it is born at stage $i$.

If $[z_i]$ became a boundary at some later stage $p$, then the boundary subspace at stage $p$ would contain a relation whose leading positive index is $i$. By the reduced matrix description in the proof of Theorem~\ref{thm:active-basis}, this would imply the existence of a negative column $R_j$ satisfying $\operatorname{low}(R_j)=i$ for some $j\leq p$. This contradicts the assumption that $i$ is unpaired. Hence, the class survives to the end of the filtration and determines the interval $[i,\infty)$.

It remains to verify that the interval cycles simultaneously give the interval decomposition. Fix a stage $p$. For every active positive index $i\in\mathcal A_q(p)$, the cycle $\rho_i$ has leading cell index $i$. If $i$ is paired, this follows from $\rho_i = R_{\operatorname{pair}(i)}$ and
\[
\operatorname{low}\bigl(R_{\operatorname{pair}(i)}\bigr) = i.
\]
If $i$ is unpaired, then
\[
\rho_i=z_i=V_i,
\]
and upper unitriangularity of $V$ implies that its leading cell index is $i$.

Hence, the active interval cycles have distinct leading indices. On the other hand, every nonzero boundary in $B_q(K^p;\Bbbk)$ has leading positive index among the indices already paired by stage $p$. These indices are disjoint from the active indices. Therefore, no nontrivial linear combination of the active $\rho_i$ can be a boundary, and the classes
\[
\{[\rho_i]_p\mid i\in\mathcal A_q(p)\}
\]
are linearly independent in $H_q(K^p;\Bbbk)$.

By Theorem~\ref{thm:active-basis}, $\lvert\mathcal A_q(p)\rvert = \dim_{\Bbbk}H_q(K^p;\Bbbk)$. Thus,
\[
\{[\rho_i]_p\mid i\in\mathcal A_q(p)\}
\]
is a basis of $H_q(K^p;\Bbbk)$.

The inclusion maps carry each $\rho_i$ to the class of the same cellular cycle for every stage during which the corresponding class is alive. For a paired index, the class persists from $i$ through $j-1$ and becomes zero at $j$; for an unpaired index, it persists indefinitely. Consequently, the persistence module decomposes as the direct sum of the interval modules determined by irs and the unpaired positive indices.

Finally, replacing a cell index by the filtration value of the corresponding cell sends $[i,j)$ to
\[
[f(\sigma_i),f(\sigma_j))
\]
and sends $[i,\infty)$ to
\[
[f(\sigma_i),\infty).
\]
If $f(\sigma_i)=f(\sigma_j)$, the resulting interval has zero length and exists only inside the auxiliary cellwise refinement.
\end{proof}

Theorems~\ref{thm:active-basis} and~\ref{thm:persistence pairing} provide two related but distinct families of cycle representatives.

The selected birth cycles $z_i=V_i$ are determined directly by the change-of-basis matrix produced by the chosen reduction. At every stage, the classes of the active selected birth cycles form a basis of homology. However, if $i$ is paired with $j$, the selected class $[z_i]$ need not itself become zero when $\sigma_j$ enters. It may instead become a linear combination of the remaining active homology classes.

The interval cycles $\rho_i$ are chosen for a different purpose. For a finite pair $(i,j)$, the choice $\rho_i=R_j$ represents a class born at $i$ that becomes a boundary exactly at $j$. Thus $\rho_i$ follows one interval generator through its entire lifetime. For an unpaired positive index, we set $\rho_i=z_i$.

The persistence pairing graph introduced below is constructed from the selected birth-cycle basis rather than from the interval-cycle basis, because the change-of-basis information contained in the columns of $V$ records additional relations among the selected representatives.

\section{Persistence Pairing Graphs}
\label{sec:persistence pairing-graphs}

Fix a homological dimension $q$. The purpose of this section is to record how the selected homology basis from Theorem~\ref{thm:active-basis} changes when one or more positive-length persistence intervals end at the same filtration value.

The basic linear algebraic phenomenon is illustrated as follows. Suppose $T:E\longrightarrow F$ is linear map. Let $(e_1,e_2,e_3)$ be a basis of $E$, and $(a,b)$ be a basis of $F$, with
\[
T(e_1)=a,
\qquad
T(e_2)=b,
\qquad\text{and}\qquad
T(e_3)=a+b.
\]
Although $e_3$ does not correspond to a basis element of $F$, its image is not zero. Instead, its coordinate vector in the basis $(a,b)$ is
\[
\begin{pmatrix}
1\\
1
\end{pmatrix}.
\]
Equivalently, $e_3-e_1-e_2\in\ker T$. 

The construction below records the support of analogous coordinate vectors for inclusion-induced maps in the filtration.

\subsection{Transition Coordinates}
\label{subsec:transition-coordinates}

Fix the reduction $R=DV$ and the resulting labeling of the nonzero $q$-th barcode intervals by their positive birth cells. Let $\mathcal V^{(q)}$ denote the resulting finite set of labeled intervals. For each $u\in\mathcal V^{(q)}$, write
\[
I_u=[b_u,d_u)
\]
for its persistence interval and let $i(u)$ denote its positive birth cell index. Thus, $b_u=f(\sigma_{i(u)})$. If $i(u)$ is paired with $j$, then $d_u=f(\sigma_j)$, whereas $d_u=\infty$ when $i(u)$ is unpaired. We denote the selected birth cycle associated with $u$ by
\[
z_u := z_{i(u)}.
\]

Only nonzero intervals of the original filtration are included in $\mathcal V^{(q)}$. Thus, a finite interval satisfies
\[
b_u<d_u,
\]
while an interval with $d_u=\infty$ is essential.

Let $t$ be a filtration value, and let
\[
i_t:K_{t^-}\hookrightarrow K_t
\]
denote the inclusion. We denote the set of $q$-th intervals that die at $t$ by
\[
\mathcal D_t^{(q)} := \left\{u\in\mathcal V^{(q)} \ \middle|\ b_u<t,\; d_u=t\right\},
\]
and denote the set of intervals that are present before $t$ and remain alive after $t$ by
\[
\mathcal S_t^{(q)} := \left\{v\in\mathcal V^{(q)} \ \middle|\ b_v<t<d_v \right\}.
\]
Hence, $\mathcal D_t^{(q)}$ consists of the intervals that disappear when passing from $K_{t^-}$ to $K_t$, whereas $\mathcal S_t^{(q)}$ consists of the intervals that cross the filtration value $t$.

Intervals satisfying
\[
b_u=d_u=t
\]
occur only in the auxiliary cellwise refinement determined by the tie-breaking order. They have zero length at the filtration value scale and are therefore not included in $\mathcal V^{(q)}$.

\begin{corollary}
\label{cor:survivor-image-basis}
The classes
\[
\left\{[z_v]_t \ \middle|\ v\in\mathcal S_t^{(q)} \right\}
\]
form a basis of $\operatorname{im}(i_t)_* \subseteq H_q(K_t;\Bbbk)$.
\end{corollary}

\begin{proof}
Let $v\in\mathcal S_t^{(q)}$. Since $b_v<t<d_v$, the selected cycle $z_v$ is already present in $K_{t^-}$ and its corresponding persistence interval remains alive after the transition at $t$. Hence,
\[
[z_v]_t = (i_t)_*[z_v]_{t^-}
\]
belongs to $\operatorname{im}(i_t)_*$.

After all cells with filtration value $t$ have been inserted, the interval indexed by $v$ is still active. Therefore, the classes
\[
\left\{[z_v]_t \ \middle|\ v\in\mathcal S_t^{(q)} \right\}
\]
are members of the selected active basis of $H_q(K_t;\Bbbk)$ from Theorem~\ref{thm:active-basis}, and are therefore linearly independent.

By the interval decomposition, the rank of the inclusion-induced map
\[
(i_t)_*: H_q(K_{t^-};\Bbbk) \longrightarrow H_q(K_t;\Bbbk)
\]
is the number of $q$-th intervals that are present on both sides of the transition. These are precisely the intervals in $\mathcal S_t^{(q)}$. Thus,
\[
\operatorname{rank}(i_t)_* = \left|\mathcal S_t^{(q)}\right|.
\]
The independent collection therefore has the same cardinality as $\operatorname{im}(i_t)_*$ and hence forms a basis of that image.
\end{proof}

For every $u\in\mathcal D_t^{(q)}$, the selected class $[z_u]_{t^-} \in H_q(K_{t^-};\Bbbk)$ is present immediately before $t$. Its image $(i_t)_*[z_u]_{t^-}$ belongs to $\operatorname{im}(i_t)_*$. By Corollary~\ref{cor:survivor-image-basis}, there are therefore unique coefficients $w_{uv}\in\Bbbk$, where $v\in\mathcal S_t^{(q)}$, such that
\begin{equation}
\label{eq:transition-relation}
(i_t)_*[z_u]_{t^-} = \sum_{v\in\mathcal S_t^{(q)}}w_{uv}[z_v]_t.
\end{equation}
We call the scalars $w_{uv}$ the \emph{transition coefficients at $t$}. Equivalently, the vector
\[
\bigl(w_{uv}\bigr)_{v\in\mathcal S_t^{(q)}}
\]
is the coordinate vector of $(i_t)_*[z_u]_{t^-}$ with respect to the survivor basis of $\operatorname{im}(i_t)_*$.

Notice that the disappearance of the interval label $u$ does not imply that the selected class $[z_u]$ maps to zero. Equation~\eqref{eq:transition-relation} records the combination of surviving selected classes to which it maps.

Immediately before $t$, the active $q$-th intervals are precisely those that either die at $t$ or survive beyond $t$. Hence, Theorem~\ref{thm:active-basis} gives the basis
\[
\left\{[z_u]_{t^-} \ \middle|\ u\in\mathcal D_t^{(q)} \right\} \cup \left\{[z_v]_{t^-} \ \middle|\ v\in\mathcal S_t^{(q)}\right\}
\]
of
$H_q(K_{t^-};\Bbbk)$.

\begin{proposition}
\label{prop:transition-kernel-basis}
For each $u\in\mathcal D_t^{(q)}$, let
\[
\kappa_u := [z_u]_{t^-} - \sum_{v\in\mathcal S_t^{(q)}}w_{uv}[z_v]_{t^-}.
\]
Then
\[
\left\{\kappa_u \ \middle|\ u\in\mathcal D_t^{(q)}\right\}
\]
is a basis of $\ker(i_t)_*$.
\end{proposition}

\begin{proof}
For every $v\in\mathcal S_t^{(q)}$, inclusion sends the class of the same selected cycle in $K_{t^-}$ to its class in $K_t$. That is, $(i_t)_*[z_v]_{t^-} = [z_v]_t$. Therefore, using~\eqref{eq:transition-relation},
\[
\begin{aligned}
(i_t)_*\kappa_u &= (i_t)_*[z_u]_{t^-} - \sum_{v\in\mathcal S_t^{(q)}}w_{uv}(i_t)_*[z_v]_{t^-}\\
&= \sum_{v\in\mathcal S_t^{(q)}}w_{uv}[z_v]_t - \sum_{v\in\mathcal S_t^{(q)}}w_{uv}[z_v]_t\\
&=0.
\end{aligned}
\]
Thus, $\kappa_u\in\ker(i_t)_*$ for every $u\in\mathcal D_t^{(q)}$.

Now, we prove linear independence. Suppose
\[
\sum_{u\in\mathcal D_t^{(q)}}a_u\kappa_u = 0.
\]
Expanding the definition of $\kappa_u$ gives
\[
\sum_{u\in\mathcal D_t^{(q)}}a_u[z_u]_{t^-} - \sum_{u\in\mathcal D_t^{(q)}}\sum_{v\in\mathcal S_t^{(q)}}a_uw_{uv}[z_v]_{t^-} = 0.
\]
Relative to the selected active basis of $H_q(K_{t^-};\Bbbk)$ displayed above, the coefficient of the distinct basis vector $[z_u]_{t^-}$ is exactly $a_u$. Hence, $a_u=0$ for every $u\in\mathcal D_t^{(q)}$. Therefore, the family $\{\kappa_u\}$ is linearly independent.

Finally, the interval decomposition shows that the kernel of $(i_t)_*: H_q(K_{t^-};\Bbbk) \longrightarrow H_q(K_t;\Bbbk)$ has one dimension for each $q$-th interval that ends at $t$. Consequently,
\[
\dim_{\Bbbk}\ker(i_t)_* = \left|\mathcal D_t^{(q)}\right|.
\]
The family
\[
\left\{\kappa_u \ \middle|\ u\in\mathcal D_t^{(q)} \right\}
\]
is an independent subset of the kernel with exactly this cardinality, so it is a basis of $\ker(i_t)_*$.
\end{proof}

Proposition~\ref{prop:transition-kernel-basis} makes the distinction between a dying persistence label and the selected cycle carrying that label precise. The kernel vector
\[
\kappa_u = [z_u]_{t^-} - \sum_{v\in\mathcal S_t^{(q)}}w_{uv}[z_v]_{t^-}
\]
is genuinely killed by the inclusion-induced map. In other words, $(i_t)_*\kappa_u=0$. The selected class $[z_u]_{t^-}$ itself need not lie in the kernel.

If $w_{uv}=0$ for every $v\in\mathcal S_t^{(q)}$, then $\kappa_u=[z_u]_{t^-}$, so the selected birth class itself becomes zero in homology at the transition. On the other hand, if $w_{uv}\neq0$ for at least one survivor $v$, then
\[
(i_t)_*[z_u]_{t^-} = \sum_{v\in\mathcal S_t^{(q)}} w_{uv}[z_v]_t \neq0,
\]
because the survivor classes are linearly independent. In this case, the label $u$ disappears from the selected active basis, but the image of its selected birth class remains nonzero and is re-expressed in the survivor basis.

The preceding construction can be summarized by the short exact sequence
\[
0 \longrightarrow \ker(i_t)_* \longrightarrow H_q(K_{t^-};\Bbbk) \xrightarrow{(i_t)_*} \operatorname{im}(i_t)_* \longrightarrow 0.
\]
Here, being exact means that the first map identifies $\ker(i_t)_*$ with the classes sent to zero by $(i_t)_*$, while the last map is onto $\operatorname{im}(i_t)_*$ by definition.

Immediately before $t$, the selected active basis is
\[
\left\{[z_u]_{t^-}\mid u\in\mathcal D_t^{(q)}\right\} \cup \left\{[z_v]_{t^-}\mid v\in\mathcal S_t^{(q)}\right\}.
\]
Proposition~\ref{prop:transition-kernel-basis} replaces the first part of this basis by the kernel basis $\left\{\kappa_u\mid u\in\mathcal D_t^{(q)}\right\}$, where
\[
\kappa_u = [z_u]_{t^-} - \sum_{v\in\mathcal S_t^{(q)}}w_{uv}[z_v]_{t^-}.
\]
Thus, the matrix of transition coefficients $W_t=(w_{uv})$ records the difference between the reduction-selected active basis and a basis adapted to the kernel--image decomposition of the inclusion map. This is the linear algebraic information later recorded by a persistence pairing graph.

\subsection{Transition Coefficients via the Kronecker Pairing}
\label{subsec:kronecker-coordinates}

The transition coefficients can be recovered from the Kronecker pairing. For brevity, we write
\[
\langle\cdot,\cdot\rangle_t
\]
for the Kronecker pairing on $K_t$.

By Corollary~\ref{cor:survivor-image-basis}, the classes
\[
\left\{[z_v]_t \ \middle|\ v\in\mathcal S_t^{(q)} \right\}
\]
form a basis of $\operatorname{im}(i_t)_*$. Extending this family to a basis, we have $\mathcal B_t$ of $H_q(K_t;\Bbbk)$. Let $\mathcal B_t^*$ be the dual basis of $H^q(K_t;\Bbbk)$ under the Kronecker pairing. For each $v\in\mathcal S_t^{(q)}$, let $\alpha_v^t\in H^q(K_t;\Bbbk)$ be the dual basis element corresponding to $[z_v]_t$. Thus,
\[
\left\langle\alpha_v^t, [z_r]_t \right\rangle_t = \delta_{vr}
\]
for all $v,r\in\mathcal S_t^{(q)}$.

\begin{proposition}
\label{prop:kronecker-coordinate-formula}
For every $u\in\mathcal D_t^{(q)}$ and $v\in\mathcal S_t^{(q)}$, the transition coefficient $w_{uv}$ satisfies
\begin{equation}
\label{eq:kronecker-coordinate}
w_{uv} = \left\langle\alpha_v^t, (i_t)_*[z_u]_{t^-}\right\rangle_t = \left\langle i_t^*\alpha_v^t, [z_u]_{t^-} \right\rangle_{t^-}.
\end{equation}
\end{proposition}

\begin{proof}
By~\eqref{eq:transition-relation},
\[
(i_t)_*[z_u]_{t^-} = \sum_{r\in\mathcal S_t^{(q)}}w_{ur}[z_r]_t.
\]
Pairing both sides with $\alpha_v^t$ gives
\[
\begin{aligned}
\left\langle \alpha_v^t, (i_t)_*[z_u]_{t^-} \right\rangle_t &= \sum_{r\in\mathcal S_t^{(q)}} w_{ur} \left\langle \alpha_v^t, [z_r]_t \right\rangle_t\\
&= \sum_{r\in\mathcal S_t^{(q)}} w_{ur}\delta_{vr}\\
&= w_{uv}.
\end{aligned}
\]
The second equality in~\eqref{eq:kronecker-coordinate} follows from Lemma~\ref{lem:kronecker-naturality}.
\end{proof}

Although the construction of $\alpha_v^t$ uses an extension of the survivor basis to a basis of the whole homology space, the resulting number $w_{uv}$ is independent of the chosen extension. Indeed, the class $(i_t)_*[z_u]_{t^-}$ belongs to $\operatorname{im}(i_t)_*$, so only the restriction of $\alpha_v^t$ to this image is used in Equation~\eqref{eq:kronecker-coordinate}. This restriction is uniquely determined by
\[
\left\langle\alpha_v^t,[z_r]_t\right\rangle_t = \delta_{vr}
\]
for all $r\in\mathcal S_t^{(q)}$. The coefficients $w_{uv}$ can nevertheless depend on the chosen survivor basis and therefore on the selected representatives produced by the fixed reduction.

Let $\mathcal R=(R,V)$ denote the chosen reduced factorization, with the cell orientations fixed throughout. We regard the data used in the construction as
\[
\mathfrak F = (K,f,\prec,\Bbbk,\mathcal R).
\]
It is useful to distinguish the information determined by the filtered complex from the information introduced by auxiliary choices:
\begin{itemize}
\item The persistence module and its barcode, up to isomorphism, are determined by $(K,f,\Bbbk)$.
\item For every filtration value $t$, the subspaces $\operatorname{im}(i_t)_*$ and $\ker(i_t)_*$ are determined intrinsically by the inclusion
\[
i_t:K_{t^-}\hookrightarrow K_t.
\]
\item When several cells have the same filtration value, their cell-level birth and death labels in the auxiliary cellwise filtration can depend on the tie-breaking order $\prec$.
\item The selected birth cycles $z_u$ depend on the chosen reduction. Consequently, the survivor basis, the transition coefficients $w_{uv}$, and any graph defined from the support of these coefficients can depend on the reduced factorization $\mathcal R$ and, in the presence of ties, on $\prec$.
\end{itemize}

The transition coefficients now define a directed graph on the labeled persistence intervals.

\begin{definition}
\label{def:persistence pairing-graph}
The $q$-th \emph{persistence pairing graph} associated with $\mathfrak F$ is the directed graph
\[
\mathcal G^{(q)}(\mathfrak F) = \bigl(\mathcal V^{(q)},\mathcal E^{(q)}\bigr),
\]
whose vertex set $\mathcal V^{(q)}$ consists of the labeled nonzero $q$-th persistence intervals and whose edge set is
\[
\mathcal E^{(q)} := \left\{(u,v) \ \middle|\ d_u<\infty, b_v<d_u<d_v, \text{ and } w_{uv}\neq 0\right\}.
\]
If the nonzero transition coefficients are retained, we define the edge weight
\[
\omega^{(q)}(u,v) := w_{uv} \in \Bbbk\setminus\{0\}.
\]
\end{definition}

An edge $u\longrightarrow v$ is present precisely when the selected birth class associated with $u$ has a nonzero coordinate in the selected survivor direction associated with $v$ when $u$ ceases to be active.

The graph records a change of coordinates in the selected homology basis. In particular, an edge $u\longrightarrow v$ does not mean that the persistence interval indexed by $u$ survives beyond $d_u$. By Theorem~\ref{thm:persistence pairing}, the interval cycle $\rho_{i(u)}$ becomes trivial in homology at $d_u$, as required by the interval decomposition. Rather, the edge records that the selected birth cycle $z_u$ has a nonzero image in the survivor direction labeled by $v$.

\begin{proposition}
\label{prop:pairing-graph-dag}
A persistence pairing graph $\mathcal G^{(q)}(\mathfrak F)$ is a directed acyclic graph.
\end{proposition}

\begin{proof}
If $u\longrightarrow v$, then, by definition, $d_u<d_v$. Hence, death values strictly increase along every directed path. A directed cycle
\[
u_1 \longrightarrow u_2 \longrightarrow \cdots \longrightarrow u_r \longrightarrow u_1
\]
would therefore imply
\[
d_{u_1} < d_{u_2} < \cdots < d_{u_r} < d_{u_1},
\]
which leads to a contradiction.
\end{proof}

The preceding proposition uses only the ordering of death values. The triangular structure of the selected cycles gives a stronger restriction on the possible edges. In fact, an edge can point only toward an interval whose birth occurs no later than the birth of the source interval and whose death occurs strictly later. Thus, directed edges respect a nesting relation among persistence intervals.

\begin{proposition}
\label{prop:edge-interval-containment}
If $u\longrightarrow v$ is an edge of $\mathcal G^{(q)}(\mathfrak F)$, then $i(v)<i(u)$. Consequently,
\[
b_v\leq b_u<d_u<d_v.
\]
In particular, $I_u\subseteq I_v$.
\end{proposition}

\begin{proof}
Let $i=i(u)$ and $k=i(v)$. Since $u\longrightarrow v$, the transition coefficient $w_{uv}$ is nonzero. At the death value $t=d_u$, we have
\[
(i_t)_*[z_u]_{t^-} = \sum_{r\in\mathcal S_t^{(q)}}w_{ur}[z_r]_t.
\]
The selected cycle $z_u=V_i$ is supported only on cells with index at most $i$, because $V$ is upper triangular.

Suppose that $w_{uv}\neq0$ and $k>i$. The selected cycle $z_v=V_k$ has leading cell index $k$. Since $v$ survives beyond $t$, its positive index $k$ has not yet been paired by the end of the transition at $t$. Therefore, no boundary present in $K_t$ has leading positive index $k$. The nonzero leading $k$-term contributed by $w_{uv}z_v$ therefore cannot be cancelled by a boundary or by selected survivor cycles with smaller leading index. This contradicts the fact that the left-hand side is represented by a chain supported on indices at most $i$.

Hence $k<i$. Since the cell order refines the filtration,
\[
b_v=f(\sigma_k)\leq f(\sigma_i)=b_u.
\]
Because $u$ is a nonzero interval dying at $d_u$, $b_u<d_u$. Finally, the definition of an edge gives $d_u<d_v$. Combining these inequalities yields
\[
b_v\leq b_u<d_u<d_v,
\]
and therefore $I_u\subseteq I_v$.
\end{proof}

Propositions~\ref{prop:pairing-graph-dag} and~\ref{prop:edge-interval-containment} describe restrictions that hold once the reduction data have been fixed. They show that the graph has an ordered structure: edges point toward classes with longer persistence intervals that contain the source interval. These statements should not, however, be read as claims of geometric containment between cycle representatives.

For $q>0$, an edge in a persistence pairing graph is an algebraic basis transition relation and need not represent a geometric merger of subsets of $K$. Homological dimension zero is special because it admits a natural basis indexed by connected components. In that case, the transition relations have a more direct interpretation in terms of component mergers.

The preceding results concern the structure of the graph for a fixed reduction. We next examine a different question: which parts of the graph are determined by the persistence module itself, and which parts depend on the basis selected by the reduction? The barcode remains unchanged under many changes of basis, but the selected birth cycles need not remain unchanged. The next result gives the simplest such change explicitly.

\subsection{Dependence on the Chosen Reduction}
\label{subsec:reduction-dependence}

The persistence module and its barcode are determined by the filtered complex, but the selected birth cycles used in a persistence pairing graph depend on the chosen reduction. Consequently, the transition coefficients and graph edges need not be determined by the barcode alone. The following results make this dependence explicit, first through a local change of one selected birth cycle and then through a choice of representatives for which all finite transition coefficients vanish.

\begin{proposition}
\label{prop:birth-cycle-change}
Let $k<i$ be positive cell indices of the same dimension such that $R_k=R_i=0$. For any $a\in\Bbbk$, let $V'$ be a matrix given by
\[
V'_i=V_i+aV_k
\]
and, for $\ell\neq i$, $V'_\ell=V_\ell$. Then $V'$ is upper unitriangular and dimension preserving, and
\[
DV'=R.
\]
Thus, $(R,V')$ has exactly the same reduced matrix and the same persistence pairs as $(R,V)$.

Suppose further that the interval $u$ has birth index $i$, dies at $t$, and that the interval $v$ with birth index $k$ survives beyond $t$. If $w_{ur}$ and $w'_{ur}$ denote the transition coefficients before and after the change, then
\[
w'_{uv}=w_{uv}+a
\]
and, for $r\neq v$, $w'_{ur}=w_{ur}$.
\end{proposition}

\begin{proof}
Because $k<i$, adding $aV_k$ to $V_i$ preserves upper unitriangularity. Since $k$ and $i$ have the same dimension, the change also preserves dimension.

Moreover,
\[
DV'_i = D(V_i+aV_k) = R_i+aR_k = 0,
\]
while every other column is unchanged. Hence, $DV'=R$. In particular, the reduced matrix, its low indices, and all persistence pairs are unchanged.

The new selected birth cycle for $u$ is
\[
z'_u=z_u+az_v.
\]
Since $v$ survives across $t$, $(i_t)_*[z_v]_{t^-}=[z_v]_t$. Therefore,
\[
\begin{aligned}
(i_t)_*[z'_u]_{t^-} &= (i_t)_*[z_u]_{t^-} + a(i_t)_*[z_v]_{t^-}\\
&= \sum_{r\in\mathcal S_t^{(q)}}w_{ur}[z_r]_t + a[z_v]_t.
\end{aligned}
\]
Uniqueness of coordinates in the survivor basis gives
\[
w'_{uv}=w_{uv}+a
\]
and leaves all other transition coefficients unchanged.
\end{proof}

The proposition shows that a simple change of one selected birth cycle can alter one transition coefficient without changing the reduced matrix or any persistence pair. Thus, the same barcode data can support different directed edge relations. The following consequence makes this dependence explicit at the level of the graph.

\begin{corollary}
\label{cor:edge-reduction-dependence}
Under the assumptions of Proposition~\ref{prop:birth-cycle-change}, the support of the transition coefficients can change while the reduced matrix, persistence pairs, and barcode remain unchanged.

In particular, if $w_{uv}\neq0$, then choosing
\[
a=-w_{uv}
\]
removes the edge $u\longrightarrow v$. If $w_{uv}=0$, then choosing any nonzero $a$ creates that edge.
\end{corollary}

The preceding corollary describes a local change. An admissible basis modification can create or remove a particular edge. The dependence on the selected basis is stronger than this local observation suggests. We can choose a basis adapted to the persistence intervals themselves so that every finite selected birth cycle becomes zero at the death of its interval. Under this choice, all transition coefficients vanish.

\begin{proposition}
\label{prop:interval-adapted-reduction}
For every paired positive index $i$, let
\[
j=\operatorname{pair}(i)
\]
and let $\gamma_i\neq0$ be the coefficient of the cell $\sigma_i$ in the reduced column $R_j$. Define
\[
\widehat V_i:=\gamma_i^{-1}R_j.
\]
For unpaired positive indices and for negative indices, leave the corresponding columns unchanged.

Then $\widehat V$ is upper unitriangular and dimension preserving, and
\[
D\widehat V=R.
\]
For every finite nonzero persistence interval, the selected birth cycle produced by $\widehat V$ becomes zero in homology at its death value. Consequently, every transition coefficient of a dying label is zero and the corresponding persistence pairing graph has no directed edges.
\end{proposition}

\begin{proof}
If $i$ is paired with $j$, then $\operatorname{low}(R_j)=i$. Hence, $R_j$ is supported on cell indices at most $i$ and has coefficient $\gamma_i\neq0$ at index $i$. Therefore, $\gamma_i^{-1}R_j$ has coefficient $1$ at $i$ and no support at indices larger than $i$. Thus replacing $V_i$ by this column preserves upper unitriangularity. The replacement also preserves dimension.

Moreover,
\[
D\widehat V_i = \gamma_i^{-1}DR_j = \gamma_i^{-1}D^2V_j = 0 = R_i.
\]
All unchanged columns still satisfy $D\widehat V_\ell=R_\ell$. Hence,
\[
D\widehat V=R.
\]

For a finite pair $(i,j)$, the selected birth cycle is now
\[
\widehat z_i=\gamma_i^{-1}R_j.
\]
By Proposition~\ref{prop:killed-cycle}, $R_j$ becomes a boundary when the cell $\sigma_j$ enters. Thus,
\[
(i_t)_*[\widehat z_i]_{t^-}=0
\]
at the death value $t$ of the corresponding interval. Its coordinates in the survivor basis are therefore all zero. Essential intervals have no finite death and hence no outgoing
edges. Thus the graph contains no directed edges.
\end{proof}

Proposition~\ref{prop:interval-adapted-reduction} gives the strongest form of the reduction dependence considered here. Even when the reduced matrix $R$, irs, and the barcode are unchanged, the selected basis can be chosen so that all directed edges disappear. Therefore, the edge set is not determined by the persistence module or by the reduced matrix alone. It records information about the particular representatives selected by the full reduction data. For this reason, we regard
\[
\mathfrak F=(K,f,\prec,\Bbbk,\mathcal R)
\]
as the complete input to a persistence pairing graph and assume that a fixed deterministic reduction convention is used whenever graphs are compared.

This dependence on the reduction is especially important in positive homological dimensions, where homology classes have no preferred cycle representatives in general. Dimension zero provides a useful contrast. Connected components give a natural geometric basis, and under the standard dimension-zero reduction convention the selected basis agrees with the usual elder rule labeling. We now show that, under these assumptions, a persistence pairing graph recovers the familiar branch relation of the merge tree.

\subsection{Relation to Merge Trees}
\label{subsec:merge-tree-relation}

We first recall the classical merge tree construction
\cite{morozov2013interleaving}. Let $g:X\longrightarrow\mathbb R$ be continuous. Assume that every sublevel set
\[
X_{\leq a} := g^{-1}((-\infty,a])
\]
is locally path-connected and has finitely many connected components, and that the component structure changes at only finitely many parameter values.

We define the epigraph of $g$ by
\[
\operatorname{epi}(g) := \left\{ (x,a)\in X\times\mathbb R \ \middle|\ g(x)\leq a \right\}.
\]
For $(x,a),(y,b)\in\operatorname{epi}(g)$, let $(x,a)\sim(y,b)$ if and only if $a=b$ and $x$ and $y$ belong to the same connected component of $X_{\leq a}$.

\begin{definition}
\label{def:merge-tree}
The \emph{merge tree} of $g$ is the quotient space
\[
M_g := \operatorname{epi}(g)/{\sim},
\]
equipped with the height function $h:M_g\longrightarrow\mathbb R$ given by
\[
h([(x,a)])=a.
\]
\end{definition}

A point of $M_g$ represents a connected component of a sublevel set at a particular parameter value. As the parameter increases, branches meet precisely when connected components merge. If the component structure eventually stabilizes to a single component, $M_g$ has one unbounded branch. If several components remain indefinitely, the same construction gives a \emph{merge forest}.

To compare the merge tree with $0$th persistent homology, we use the classical elder rule \cite{edelsbrunnerharer2010computational}. When two or more components merge, the branch with the earliest birth value is retained and the remaining branches terminate. If several branches have the same birth value, their relative order is determined by the same tie-breaking order $\prec$ used in the filtered boundary matrix.

For this comparison, let $\mathcal B_g$ be the directed graph whose vertices are the nonzero elder rule branches. When a nonzero branch terminates at a filtration value $t$, draw an edge from that branch to the unique nonzero branch labeling the connected component containing it immediately after the transition at $t$. Thus, simultaneous multiway merges produce one edge from each terminating nonzero branch to the branch retained after the entire merge event.

Now, assume that the component evolution of the cellular filtration $\{K_t\}$ agrees with that of the sublevel filtration $\{X_{\leq t}\}$ at the filtration values under consideration. Also, we assume that the $0$th reduction is \emph{zero-column preserving}, that is, for every positive $0$-cell $\sigma_i$,
\[
V_i=e_i.
\]
Equivalently, $z_i=\sigma_i$. The standard left-to-right reduction has this property because the boundary column of every $0$-cell is zero and therefore requires no column operations.

\begin{proposition}
\label{prop:h0-pairing-graph-merge-tree}
Under the preceding assumptions and a common tie-breaking order, the $0$th persistence pairing graph $\mathcal G^{(0)}(\mathfrak F)$ is isomorphic to $\mathcal B_g$. Every edge has coefficient $1\in\Bbbk$. Consequently, if the component structure eventually stabilizes to one component, $\mathcal G^{(0)}(\mathfrak F)$ is a rooted directed tree whose edges are oriented toward its unique root. If several terminal components remain, it is a rooted directed forest with one root for each terminal component.
\end{proposition}

\begin{proof}
Because every sublevel set is locally path-connected, its connected components agree with its path components. Hence, $H_0(X_{\leq a};\Bbbk)$ has a natural basis indexed by the connected components of $X_{\leq a}$.

For $a\leq b$, the inclusion $X_{\leq a}\hookrightarrow X_{\leq b}$ sends the homology class of a connected component of $X_{\leq a}$ to the class of the unique connected component of $X_{\leq b}$ that contains it.

Under the zero-column preserving assumption, every selected $0$th birth cycle is a single positive $0$-cell. In other words,
\[
z_i=\sigma_i.
\]
At every stage, the active selected classes therefore give one distinguished vertex class for each connected component. The standard $0$th persistence pairing retains the component with the earliest birth value when components merge. When birth values agree, the order $\prec$ determines which label is retained. Thus, the selected active basis agrees with the component basis labeled by the elder rule.

Let $u$ be a nonzero branch that terminates at a filtration value $t$, and let $v$ be the branch retained in the connected component containing it after the transition. Immediately before $t$, the selected vertices $z_u$ and $z_v$ lie in distinct connected components. Immediately after $t$, they lie in the same connected component. Therefore,
\[
(i_t)_*[z_u]_{t^-} = [z_v]_t.
\]
Since the survivor component classes form a basis of $\operatorname{im}(i_t)_*$, the transition coordinates of $u$ satisfy
\[
w_{uv}=1
\]
and $w_{ur}=0$ for every $r\in\mathcal S_t^{(0)}$ with $r\neq v$. Hence, $u\longrightarrow v$ is an edge of $\mathcal G^{(0)}(\mathfrak F)$ exactly when the
elder rule branch $u$ terminates into the component whose retained branch is $v$. This gives an isomorphism
\[
\mathcal G^{(0)}(\mathfrak F) \cong \mathcal B_g,
\]
and every edge weight is equal to $1$.

Every finite nonzero elder rule branch terminates exactly once, whereas each component that survives indefinitely contributes one unbounded branch. It follows that the directed graph is a rooted tree when there is one terminal component and a rooted forest when there are several.
\end{proof}

When birth values are repeated, the unlabeled merge tree does not by itself determine which equal-birth branch is retained at a merge. Accordingly, the labeled comparison above depends on the common tie-breaking convention.

\subsection{Computation}
\label{subsec:computation}

A persistence pairing graph can be computed from a standard persistence reduction \cite{edelsbrunner2002topological,zomorodian2005computing}, with transformation data retained to recover selected cycle representatives \cite{hang2021umatch}. A practical workflow is as follows.

\begin{enumerate}
\item Order the cells by filtration value, subject to the condition that every proper face precedes the cell containing it, and fix a deterministic rule for all remaining ties.
\item Reduce the boundary matrix to obtain $R=DV$. The low indices of the nonzero columns determine irs, while every positive column $V_i$ determines the selected birth cycle $z_i=V_i$.
\item For every filtration value $t$ at which at least one nonzero $q$-th interval dies, determine the sets $\mathcal D_t^{(q)}$ and $\mathcal S_t^{(q)}$.
\item For each $u\in\mathcal D_t^{(q)}$, express $(i_t)_*[z_u]_{t^-}$ in the survivor basis $\left([z_v]_t\right)_{v\in\mathcal S_t^{(q)}}$. If
\[
(i_t)_*[z_u]_{t^-} = \sum_{v\in\mathcal S_t^{(q)}}w_{uv}[z_v]_t,
\]
add the directed edge $u\longrightarrow v$ whenever $w_{uv}\neq0$.
\end{enumerate}

The final step can be performed by direct homology linear algebra. Alternatively, persistent cohomology provides an equivalent dual computational viewpoint \cite{desilva2011dualities}, and the transition coefficients can be recovered through the Kronecker pairing using Proposition~\ref{prop:kronecker-coordinate-formula}. We now describe this second approach in a form that does not require the homology and cohomology computations to return mutually dual bases.

Suppose $\dim_{\Bbbk}H_q(K_t;\Bbbk)=r$. Choose a homology basis $([h_1],\ldots,[h_r])$ whose first
\[
s = \left|\mathcal S_t^{(q)}\right|
\]
elements are the selected survivor classes. Thus, after indexing the survivors as $v_1,\ldots,v_s$,
\[
[h_b]=[z_{v_b}]_t,
\]
for $1\leq b\leq s$. A natural choice for the remaining basis elements is given by the selected active classes born at $t$, although any extension to a basis of $H_q(K_t;\Bbbk)$ may be used.

Let $(\beta_1,\ldots,\beta_r)$ be any basis of $H^q(K_t;\Bbbk)$. We define the pairing matrix $P\in\Bbbk^{r\times r}$ by
\[
P_{ab} := \left\langle\beta_a,[h_b]\right\rangle_t.
\]
Because the Kronecker pairing is nondegenerate over the field $\Bbbk$, the matrix $P$ is invertible.

For any class $[z]\in H_q(K_t;\Bbbk)$, let $y_a := \left\langle\beta_a,[z] \right\rangle_t$. If
\[
[z] = \sum_{b=1}^r c_b[h_b],
\]
then
\[
y_a = \sum_{b=1}^rP_{ab}c_b.
\]
In matrix form, $y=Pc$, and therefore $c=P^{-1}y$.

Now take $[z] = (i_t)_*[z_u]_{t^-}$ for some $u\in\mathcal D_t^{(q)}$. Since this class belongs to
\[
\operatorname{im}(i_t)_* = \operatorname{span}_{\Bbbk}\left\{[z_{v_1}]_t,\ldots,[z_{v_s}]_t\right\},
\]
its coordinates in the complementary basis directions vanish. In other words,
\[
c_{s+1} = \cdots = c_r = 0.
\]
The transition coefficients are therefore $w_{u v_b}=c_b$ for $1\leq b\leq s$.

Equivalently, we define
\[
\varphi_b := \sum_{a=1}^r(P^{-1})_{ba}\beta_a,
\]
for $1\leq b\leq r$. Then
\[
\left\langle\varphi_b,[h_c]\right\rangle_t = \delta_{bc},
\]
so $(\varphi_1,\ldots,\varphi_r)$ is the cohomology basis dual to the chosen homology basis.

For an implementation, each cohomology class $\beta_a$ may be represented by a cocycle $\widetilde\beta_a$. The entries of the pairing matrix and the vector $y$ can then be evaluated directly as $P_{ab} = \widetilde\beta_a(h_b)$ and $y_a = \widetilde\beta_a(z_u)$. The latter equality is valid because the inclusion carries the same cellular cycle $z_u$ from $K_{t^-}$ into $K_t$. Thus, one need not construct a separate chain representative of $(i_t)_*[z_u]_{t^-}$.

This formulation does not require homology and cohomology reductions to produce mutually dual representatives. It requires only cycle and cocycle representatives, evaluation of cocycles on cycles, and the solution of an invertible linear system.

\subsubsection{Cubical Pair-of-Pants Example}
\label{subsubsec:pair-of-pants-example}

A finite cubical complex is a finite collection of elementary cubes that is closed under taking faces. In dimension two, its cells consist of vertices, edges, and squares \cite{wagner2012cubical,bleile2022digital}.

Let $(f_{ij})$ be a finite scalar array, and associate a square $Q_{ij}$ with each array entry. Assign to each square the filtration value $\lambda(Q_{ij}) = f_{ij}$. For every lower-dimensional face $\tau$, suppose
\[
\lambda(\tau) := \min\left\{f_{ij} \ \middle|\ \tau\subseteq Q_{ij}\right\}.
\]
If $\tau$ is a face of $Q_{ij}$, then $\lambda(\tau) \leq \lambda(Q_{ij})$, so the face monotonicity condition is satisfied. Therefore
\[
K_a := \left\{\sigma \ \middle|\ \lambda(\sigma)\leq a \right\}
\]
is a cubical subcomplex for every $a$ and defines a sublevel filtration.

Figure~\ref{fig:pair-of-pants-filtration} shows a two-dimensional example. Annular components appear successively at filtration values $0$, $10$, and $20$. Thus, immediately before filtration value $30$, there are three connected components, each carrying one independent $1$st class. At value $30$, additional squares connect the three components into a pair-of-pants region. This region is connected and has $\dim_{\mathbb F_2}H_1=2$. The two remaining one-dimensional holes are subsequently filled at filtration values $80$ and $90$.

\begin{figure}[ht]
\centering
\begin{subfigure}{0.3\textwidth}
\centering
\includegraphics[width=\linewidth]{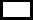}
\caption{$a=0$}
\end{subfigure}
\hfill
\begin{subfigure}{0.3\textwidth}
\centering
\includegraphics[width=\linewidth]{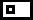}
\caption{$a=10$}
\end{subfigure}
\hfill
\begin{subfigure}{0.3\textwidth}
\centering
\includegraphics[width=\linewidth]{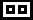}
\caption{$a=20$}
\end{subfigure}

\medskip

\begin{subfigure}{0.3\textwidth}
\centering
\includegraphics[width=\linewidth]{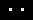}
\caption{$a=30$}
\end{subfigure}
\hfill
\begin{subfigure}{0.3\textwidth}
\centering
\includegraphics[width=\linewidth]{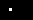}
\caption{$a=80$}
\end{subfigure}
\hfill
\begin{subfigure}{0.3\textwidth}
\centering
\includegraphics[width=\linewidth]{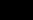}
\caption{$a=90$}
\end{subfigure}

\caption{Sublevel filtration of the cubical pair-of-pants example. Black cells belong to the cubical subcomplex at the indicated filtration value.}
\label{fig:pair-of-pants-filtration}
\end{figure}

We compute persistence with coefficients in $\mathbb F_2 = \{0, 1\}$. The nonzero $0$th intervals are
\[
[0,\infty), \qquad [10,30), \qquad\text{and}\qquad [20,30).
\]
At filtration value $30$, the components born at $10$ and $20$ join the component born at $0$. Under the elder rule, the oldest component is retained. Hence the $0$ th persistence pairing graph contains the edges 
\[
[10,30) \longrightarrow [0,\infty) \quad\text{and}\quad [20,30) \longrightarrow [0,\infty),
\]
each with coefficient $1\in\mathbb F_2$.

For the dimension-one reduction used in this example, the nonzero intervals are
\[
[0,90),\qquad [10,80),\qquad\text{and}\qquad [20,30).
\]
Immediately before filtration value $30$, there are three independent dimension-one classes. Immediately afterward, the pair-of-pants region has dimension-one homology of dimension two.

In the selected basis produced by the fixed reduction, the image of the selected birth class labeled by $[20,30)$ satisfies
\[
(i_{30})_*[z_{20}]_{30^-} = [z_{10}]_{30} + [z_0]_{30}.
\]
Thus, over $\mathbb F_2$, the two nonzero transition coefficients are both equal to $1$, producing the edges
\[
[20,30) \longrightarrow [10,80) \quad\text{and}\quad [20,30) \longrightarrow [0,90).
\]

At filtration value $80$, the corresponding selected-basis relation is $(i_{80})_*[z_{10}]_{80^-} = [z_0]_{80}$, and hence
\[
[10,80) \longrightarrow [0,90)
\]
is also an edge.

At filtration value $90$, $\mathcal S_{90}^{(1)} = \varnothing$, so the selected birth class labeled by $[0,90)$ maps to zero and the corresponding vertex has no outgoing edge.

Figure~\ref{fig:pair-of-pants-graphs} displays the resulting persistence pairing graphs by drawing each vertex as its associated persistence interval and superimposing the directed graph edges at the corresponding death values.

\begin{figure}[ht]
\centering
\includegraphics[width=0.8\textwidth]{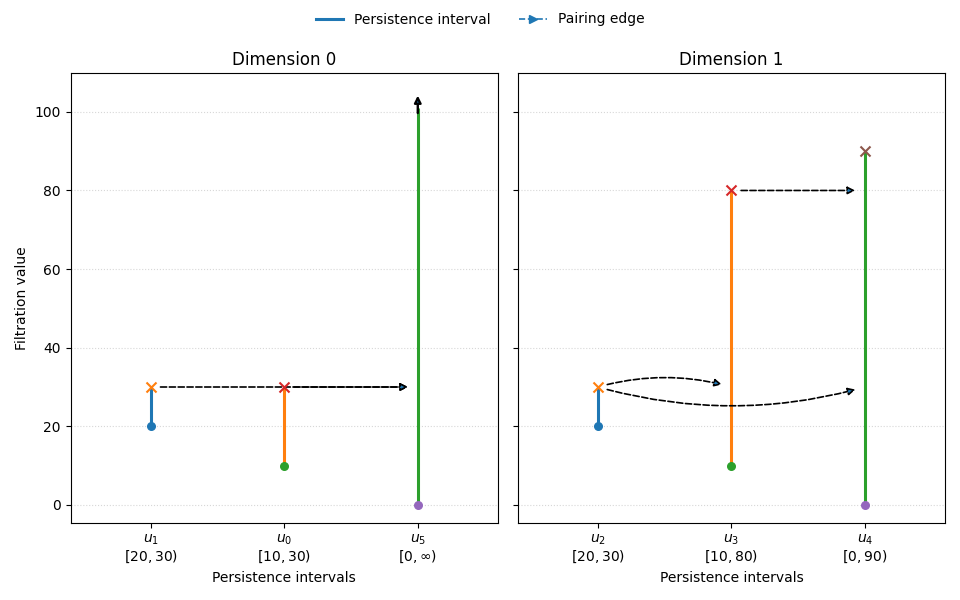}
\caption{Persistence intervals with the persistence pairing graph edges superimposed for the cubical pair-of-pants filtration.}
\label{fig:pair-of-pants-graphs}
\end{figure}

The dimension-one relation at filtration value $30$ is a relation among the selected basis classes. It does not mean that the persistence interval $[20,30)$ survives beyond $30$. By Theorem~\ref{thm:persistence pairing}, its interval representative becomes zero in homology at $30$. The equation
\[
(i_{30})_*[z_{20}]_{30^-} = [z_{10}]_{30} + [z_0]_{30}
\]
instead says that the particular selected birth class carrying the label $[20,30)$ remains nonzero after the transition and is re-expressed as the sum of the two surviving selected basis classes.

\subsection{Fused Graph Discrepancy}
\label{subsec:graph-discrepancies}

We now define a bounded discrepancy between two persistence pairing graphs. The construction is of fused Gromov--Wasserstein type: one term compares attributes attached to individual vertices, while a quadratic term compares pairwise graph structure \cite{vayer2020fused}. Since a persistence pairing graph depends on the selected representatives produced by the reduction, we call the resulting quantity a \emph{discrepancy} rather than a metric on persistence modules.

Fix a homological dimension $q$ and consider two persistence pairing
graphs
\[
\mathcal G_f^{(q)} = \bigl(V_f^{(q)},E_f^{(q)}\bigr) \quad\text{and}\quad \mathcal G_g^{(q)} = \bigl(V_g^{(q)},E_g^{(q)}\bigr).
\]
Each vertex $u$ is decorated by its persistence interval $I_u=[b_u,d_u)$.

Choose a scale $s_q>0$. The scale is fixed across all comparisons whose discrepancy values are to be compared numerically. For example, in a dataset one may choose $s_q$ once from the pooled collection of positive finite lifetimes, using a robust statistic such as a prescribed upper quantile.

We define the normalized vertex cost by
\begin{equation}
\label{eq:bounded-node-cost}
\widetilde c_{\mathrm{node}}(u,v) :=
\begin{cases}
\displaystyle
\min\left\{\frac{\max\{|b_u-b_v|,\lvert d_u-d_v\rvert\}}{s_q}, 1\right\}, & d_u,d_v<\infty, \\[4mm]
\displaystyle
\min\left\{\frac{|b_u-b_v|}{s_q}, 1\right\}, & d_u=d_v=\infty, \\[3mm]
1, & \text{exactly one of }d_u,d_v\text{ is infinite}. 
\end{cases}
\end{equation}
Observe that $0 \leq \widetilde c_{\mathrm{node}}(u,v) \leq 1$. The value $1$ assigned to a finite--essential mismatch avoids an infinite transport cost when a finite interval is matched with an essential interval.

For the unweighted directed edge structure, let
\[
A_f^{(q)}(u,u') :=
\begin{cases}
1, & (u,u')\in E_f^{(q)}, \\
0, & \text{otherwise},
\end{cases}
\]
and define $A_g^{(q)}$ analogously. The discrepancy considered here uses only the support of the edge relation. In particular, it does not require a numerical comparison of the coefficients $w_{uv}$ and is therefore meaningful over an arbitrary coefficient field.

Suppose that both vertex sets are nonempty. Let $\mu_f$ and $\mu_g$ be strictly positive probability measures on $V_f^{(q)}$ and $V_g^{(q)}$, respectively. Uniform measures are the simplest choice. Other weighting schemes may be used, provided that the same rule is fixed in advance.

A \emph{coupling} of $\mu_f$ and $\mu_g$ is a nonnegative matrix
\[
T = (T_{uv})_{u\in V_f^{(q)}, \,v\in V_g^{(q)}}
\]
whose marginals satisfy
\[
\sum_{v\in V_g^{(q)}}T_{uv} = \mu_f(u)
\]
for every $u\in V_f^{(q)}$, and
\[
\sum_{u\in V_f^{(q)}}T_{uv} = \mu_g(v)
\]
for every $v\in V_g^{(q)}$. The set of all such couplings is denoted by $\Pi(\mu_f,\mu_g)$.

For a parameter $0\leq\alpha\leq1$, we define
\begin{align}
\mathcal F_\alpha^{(q)}(T) &:= (1-\alpha)\sum_{u\in V_f^{(q)}}\sum_{v\in V_g^{(q)}}T_{uv} \widetilde c_{\mathrm{node}}(u,v)\notag\\
&\quad+ \alpha\sum_{u,u'\in V_f^{(q)}}\sum_{v,v'\in V_g^{(q)}}T_{uv}T_{u'v'}\left|A_f^{(q)}(u,u') - A_g^{(q)}(v,v')\right|.
\label{eq:fused-objective}
\end{align}

The first term measures disagreement between persistence interval attributes under the coupling. The second term measures disagreement between directed edge relations. More precisely, if $(U,V)$ and $(U',V')$ are drawn independently according to the coupling $T$, then the structural term is
\[
\mathbb E\left[\left|A_f^{(q)}(U,U') - A_g^{(q)}(V,V')\right|\right].
\]
Briefly, it is the probability, under the coupling, that the two graphs disagree about whether the corresponding directed edge is present.

\begin{definition}
\label{def:fpg-discrepancy}
Suppose both persistence pairing graphs are nonempty. For fixed $\alpha$, $s_q$, $\mu_f$, and $\mu_g$, the \emph{$q$-th fused persistence pairing graph discrepancy} is given by
\[
d_{\mathrm{FPG}}^{(q)}\left(\mathcal G_f^{(q)}, \mathcal G_g^{(q)}\right) := \min_{T\in\Pi(\mu_f,\mu_g)}\mathcal F_\alpha^{(q)}(T).
\]
\end{definition}

For empty graphs, we adopt the conventions $d_{\mathrm{FPG}}^{(q)}(\varnothing,\varnothing) := 0$ and
\[
d_{\mathrm{FPG}}^{(q)}(\varnothing,\mathcal G) = d_{\mathrm{FPG}}^{(q)}(\mathcal G,\varnothing) := 1
\]
whenever $\mathcal G$ is nonempty.

\begin{proposition}
\label{prop:fpg-basic-properties}
\label{lem:fpg-bounded}
For every admissible coupling $T$,
\[
0 \leq \mathcal F_\alpha^{(q)}(T) \leq 1.
\]
Consequently, for nonempty persistence pairing graphs,
\[
0 \leq d_{\mathrm{FPG}}^{(q)} \leq 1.
\]
The same bounds hold for the empty graph cases under the conventions specified above.
\end{proposition}

\begin{proof}
Both terms in~\eqref{eq:fused-objective} are nonnegative. Since $0 \leq \widetilde c_{\mathrm{node}}(u,v) \leq 1$ and every coupling has total mass
\[
\sum_{u,v}T_{uv}=1,
\]
the vertex term lies in $[0,1]$.

Similarly,
\[
\left|A_f^{(q)}(u,u') - A_g^{(q)}(v,v') \right| \leq1.
\]
Since
\[
\sum_{u,u',v,v'} T_{uv}T_{u'v'} = \left(\sum_{u,v}T_{uv}\right)^2 = 1,
\]
the structural term also lies in $[0,1]$. Hence,
\[
0 \leq \mathcal F_\alpha^{(q)}(T) \leq (1-\alpha)+\alpha = 1
\]
for every coupling $T$, and therefore $0 \leq d_{\mathrm{FPG}}^{(q)} \leq 1$.

The coupling set $\Pi(\mu_f,\mu_g)$ is nonempty, closed, and bounded in a finite-dimensional Euclidean space, and is therefore compact. The function $\mathcal F_\alpha^{(q)}$ is continuous in $T$. Consequently, it attains its minimum on $\Pi(\mu_f,\mu_g)$.
\end{proof}

When probability measures are used, the discrepancy compares normalized vertex mass rather than absolute vertex multiplicity. Consequently, distinct graphs can have zero discrepancy. For example, an isolated vertex and several identical isolated vertices carrying the same interval decoration can have zero discrepancy under uniform measures. Thus $d_{\mathrm{FPG}}^{(q)}$ is not asserted to be a metric on persistence pairing graphs. If absolute vertex multiplicity is to be penalized, one may supplement the discrepancy with a cardinality term or consider an unbalanced optimal transport formulation \cite{chizat2018scaling}.

The boundedness is purely a property of the normalization. The next result identifies the part of the discrepancy that inherits the usual stability of persistent homology.

\begin{lemma}
\label{lem:node-stability}
Let $f,g:\mathcal K\longrightarrow\mathbb R$ be filtration functions on the same finite CW complex, and suppose
\[
\|f-g\|_\infty \leq \varepsilon.
\]
Then the $q$-th persistence diagrams admit a bottleneck matching of cost at most $\varepsilon$.

If an off-diagonal interval $u$ is matched to an off-diagonal interval $M(u)$, then
\[
\widetilde c_{\mathrm{node}}(u,M(u)) \leq \min\left\{\frac{\varepsilon}{s_q}, 1\right\}.
\]
The same bound holds when an essential interval is matched to an essential interval.
\end{lemma}

\begin{proof}
For every filtration value $t$, the inequality $\|f-g\|_\infty\leq\varepsilon$ implies
\[
K_f(t)\subseteq K_g(t+\varepsilon) \qquad\text{and}\qquad K_g(t)\subseteq K_f(t+\varepsilon).
\]
Thus, the associated persistence modules are $\varepsilon$-interleaved. By the stability theorem for persistence diagrams \cite{cohensteiner2007stability},
\[
d_B\bigl(\operatorname{Dgm}_q(f),\operatorname{Dgm}_q(g)\bigr) \leq\varepsilon.
\]

Therefore, a bottleneck matching $M$ can be chosen so that, whenever two finite off-diagonal intervals are matched,
\[
\max\left\{|b_u-b_{M(u)}|, |d_u-d_{M(u)}|\right\} \leq \varepsilon.
\]
Substituting into~\eqref{eq:bounded-node-cost} gives
\[
\widetilde c_{\mathrm{node}}(u,M(u)) \leq \min\left\{ \frac{\varepsilon}{s_q}, 1 \right\}.
\]

For matched essential intervals, their birth values differ by at most $\varepsilon$, and the same conclusion follows from the second case of~\eqref{eq:bounded-node-cost}.
\end{proof}

A bottleneck matching may match short finite intervals to the diagonal. To isolate genuine vertex-to-vertex matches, and optionally to remove short intervals, fix a threshold $\tau\geq0$. For a bottleneck matching $M$, let
\[
R_M^{(q)}(\tau) := \left\{u\in V_f^{(q)} \ \middle|\ M(u)\in V_g^{(q)},\; \operatorname{pers}(u)>\tau,\; \operatorname{pers}(M(u))>\tau \right\},
\]
where $\operatorname{pers}([b,d)) := d-b$ for $d<\infty$, and essential intervals have infinite persistence.

Here, persistence intervals are regarded with multiplicity, so a bottleneck matching matches individual interval occurrences even when several intervals have identical endpoints.

Recall that, under the usual bottleneck-distance convention, a finite interval of persistence $\ell$ has cost $\ell/2$ to the diagonal \cite{cohensteiner2007stability,edelsbrunnerharer2010computational}. Consequently, intervals of persistence greater than $2\varepsilon$ cannot be sent to the diagonal by a matching of bottleneck cost at most $\varepsilon$.

Let $\mathcal G|_S$ denote the directed subgraph induced by a vertex set $S$. For a probability measure $\mu$ on $R_M^{(q)}(\tau)$, suppose
\begin{equation}
\label{eq:edge-disagreement}
\Delta_{E,\mu}^{(q)}(M) := \sum_{u,u'\in R_M^{(q)}(\tau)} \mu(u)\mu(u') \left| A_f^{(q)}(u,u') - A_g^{(q)}(M(u),M(u')) \right|.
\end{equation}

Equivalently, if $U$ and $U'$ are drawn independently according to $\mu$, then
\[
\Delta_{E,\mu}^{(q)}(M) = \mathbb E \left[ \left| A_f^{(q)}(U,U') - A_g^{(q)}(M(U),M(U')) \right| \right].
\]
Thus, $\Delta_{E,\mu}^{(q)}(M)$ is the probability that a randomly chosen ordered pair of matched vertices has different directed edge status in the two induced graphs.

\begin{theorem}
\label{thm:matched-subgraph-bound}
Assume $R_M^{(q)}(\tau)\neq\varnothing$. Equip $R_M^{(q)}(\tau)$ with a probability measure $\mu$ and its matched image with the pushforward measure $M_\#\mu$.
Using these measures and the same parameters $\alpha$ and $s_q$ to compute the fused discrepancy between the two induced subgraphs, we have
\[
d_{\mathrm{FPG}}^{(q)}\Bigl(\mathcal G_f^{(q)}|_{R_M^{(q)}(\tau)}, \mathcal G_g^{(q)}|_{M(R_M^{(q)}(\tau))}\Bigr) \leq (1-\alpha) \min\left\{\frac{\varepsilon}{s_q}, 1 \right\}
+ \alpha \Delta_{E,\mu}^{(q)}(M).
\]
\end{theorem}

\begin{proof}
The deterministic coupling induced by $M$ is given by
\[
T_M(u,v) :=
\begin{cases}
\mu(u), & v=M(u),\\
0, & \text{otherwise}.
\end{cases}
\]
Its marginals are $\mu$ and $M_\#\mu$.

By Lemma~\ref{lem:node-stability},
\[
\begin{aligned}
\sum_{u,v} T_M(u,v) \widetilde c_{\mathrm{node}}(u,v) &= \sum_u \mu(u) \widetilde c_{\mathrm{node}}(u,M(u))\\
&\leq\min\left\{\frac{\varepsilon}{s_q}, 1\right\}\sum_u\mu(u)\\
&= \min\left\{\frac{\varepsilon}{s_q}, 1\right\}.
\end{aligned}
\]

Substituting the same coupling into the structural term gives
\[
\sum_{u,u'}\mu(u)\mu(u')\left|A_f^{(q)}(u,u') - A_g^{(q)}(M(u),M(u'))\right| = \Delta_{E,\mu}^{(q)}(M).
\]
Therefore,
\[
\mathcal F_\alpha^{(q)}(T_M) \leq (1-\alpha)\min\left\{\frac{\varepsilon}{s_q}, 1\right\} + \alpha\Delta_{E,\mu}^{(q)}(M).
\]
Since the fused discrepancy is the infimum over all admissible couplings, it cannot exceed its value at $T_M$.
\end{proof}

\begin{corollary}
\label{cor:edge-preserved-bound}
Under the assumptions of Theorem~\ref{thm:matched-subgraph-bound}, suppose additionally that the matching preserves directed adjacency on $R_M^{(q)}(\tau)$; that is,
\[
A_f^{(q)}(u,u') = A_g^{(q)}(M(u),M(u'))
\]
for all $u,u'\in R_M^{(q)}(\tau)$. Then $\Delta_{E,\mu}^{(q)}(M)=0$ and
\[
d_{\mathrm{FPG}}^{(q)}\Bigl(\mathcal G_f^{(q)}|_{R_M^{(q)}(\tau)}, \mathcal G_g^{(q)}|_{M(R_M^{(q)}(\tau))}\Bigr) \leq (1-\alpha)\min\left\{\frac{\varepsilon}{s_q}, 1 \right\}.
\]
\end{corollary}

Theorem~\ref{thm:matched-subgraph-bound} is intentionally not stated as a stability theorem for the full persistence pairing graph. The first term is controlled by the ordinary stability of persistence intervals. The second term measures the actual disagreement of the reduction-dependent directed edge relations under the chosen bottleneck matching.

Moreover, the theorem concerns only the induced subgraphs on the retained matched vertices. Intervals matched to the diagonal, intervals removed by the persistence threshold, and edges incident to excluded vertices are not represented in the bound.

In higher homological dimensions, a small perturbation of the filtration can change the selected representatives and hence the graph edges even when the underlying persistence diagrams remain close. The term $\Delta_{E,\mu}^{(q)}(M)$ makes this possible structural change explicit rather than assuming that it is controlled by $\|f-g\|_\infty$.

\section{Conclusions}

We introduced persistence pairing graphs to record information that is intentionally absent from a persistence barcode. A barcode records the lifetime of interval summands. After a matrix reduction has been fixed, however, one also obtains concrete selected birth cycles. When a persistence label ends, the selected cycle carrying that label need not itself become zero in homology. It may instead map to a linear combination of selected classes whose labels survive longer. A persistence pairing graph records the support of these transition coordinates.

A central point of the paper is the separation between the selected birth cycle and an interval-cycle representative. For a positive index $i$, the transformation column $V_i$ gives the selected birth cycle $z_i$. The active classes $[z_i]$ form a basis of homology at every cellwise stage. If $i$ is paired with a negative index $j$, the reduced death column $R_j$ gives a different cycle. It represents the interval born at $i$ and becomes a boundary exactly when $j$ enters. This distinction explains why a graph edge does not mean that a persistence interval survives past its death. The interval dies as usual, however, the edge records how the reduction-selected basis is rewritten across the same transition.

At each filtration value $t$, the selected survivor classes form a basis of $\operatorname{im}(i_t)_*$. Hence every selected label that dies at $t$ has a unique coordinate vector in this image basis. The associated kernel vectors form a basis of $\ker(i_t)_*$, and the same transition coefficients can be recovered by the Kronecker pairing with cohomology classes dual to the survivor basis. This gives both a homological and a cohomological description of the same basis transition.

The resulting graph has a simple global structure: every edge points from an interval that dies earlier to one that survives longer, so the graph is directed and acyclic. In dimension zero, when the standard reduction keeps positive vertices unchanged, the construction agrees with the elder rule branch relation of the merge tree. In higher dimensions, the meaning is algebraic rather than a literal merger of geometric subsets.

The dependence on the chosen reduction is not a defect to hide but part of the definition. The persistence module, its barcode, and the image and kernel of each inclusion map are determined by the filtered complex. The selected cycles and their coordinate supports need not be. For this reason, a persistence pairing graph should be interpreted as additional information attached to a fixed cell order and reduction convention. It is not an invariant determined by the persistence module alone.

We also introduced a bounded fused graph discrepancy that compares both persistence interval attributes and directed edge structure. The comparison theorem makes the same distinction explicit: ordinary persistence stability controls the interval contribution, while disagreement of graph edges remains a separate term. This prevents a stability statement about persistence diagrams from being incorrectly transferred to reduction-dependent graph edges.

Several questions remain open. On the theoretical side, it is natural to seek summaries that retain some transition information while depending less strongly on the chosen basis, for example through subspaces, ranks, or equivalence classes under allowed basis changes. It is also important to identify conditions under which particular edge relations are preserved by a fixed deterministic reduction rule. On the computational side, large image filtrations require methods that avoid storing the full change-of-basis matrix while still recovering the graph exactly. Finally, applications should test whether the additional edge information explains structure that is not already captured by persistence intervals alone.

The main message is therefore simple. A persistence barcode describes \emph{when} homological features live. A persistence pairing graph, after a reduction has been fixed, describes one precise way in which the selected homology basis changes when those labels end. These are different questions, and keeping them separate makes the added information and its limitations clear.

\bibliographystyle{plainnat}
\bibliography{main}

\end{document}